\documentclass[11pt,leqno]{amsart}

\usepackage{epsfig}
\usepackage[utf8]{inputenc}

  \usepackage{url}
 \usepackage{mathtools}
\mathtoolsset{showonlyrefs}

 \usepackage[T1]{fontenc}

 \usepackage{amssymb}

    \newtheorem{theorem}{Theorem}
    
    \newtheorem{proposition}{Proposition}
    
    \newtheorem{lemma}{Lemma}
    \newtheorem{remark}{Remark}
    \newtheorem{corollary}{Corollary}
    
    \newtheorem{definition}{Definition}
\newcommand{\eps}{\varepsilon}

\newcommand{\C}{\mathcal{C}}
\newcommand{\A}{\mathcal{A}}

\newcommand{\D}{\mathcal{D}}

    \newcommand\CC{\hbox{C\kern -.58em {\raise .54ex \hbox
    {$\scriptscriptstyle |$}}
    \kern-.55em {\raise .53ex \hbox{$\scriptscriptstyle |$}} }}
      \newcommand\qd{\hfill$\sqcap\kern-8.0pt\hbox{$\sqcup$}$}
    \newcommand\NN{\hbox{I\kern-.2em\hbox{N}}}
       \newcommand\nn{\hbox{I\kern-.2em\hbox{N}}}
    \newcommand\RR{I\!\!R}
    \newcommand\sRR{{\sl \hbox{I\kern-.2em\hbox{R}}}}
    \newcommand\QQ{\hbox{I\kern-.53em\hbox{Q}}}
     
\newcommand\sign{\hbox{Sign}}

\newcommand\signp{{\hbox{Sign}^+}}
    
 \newcommand\spo{\hbox{Sign}_0^+}
 \newcommand \smo{\hbox{Sign}_0^-}
    \newcommand\heps{{\mathcal H}_\varepsilon}
  \newcommand\hepsp{{\mathcal H}_\varepsilon^+}

 \everymath{\displaystyle}

\usepackage{stackengine}
\usepackage{scalerel}
\usepackage{xcolor,amssymb}

\begin{document}
\title[$L^1$-Theory for Hele-Shaw Flow]
{$L^1$-Theory for  Hele-Shaw Flow   with linear Drift}
 \newcommand{\cs}{$^\dagger$} \newcommand{\cm}{$^\ddagger$}

\author[N. Igbida]{Noureddine Igbida\cs}

\thanks{\cs  Institut de recherche XLIM, UMR-CNRS 7252, Facult\'e des Sciences et Techniques,  Universit\'e de Limoges, France. E-mail : {\tt noureddine.igbida@unilim.fr}}


\newcommand{\nfont}{\fontshape{n}\selectfont}


\date{\today}

\maketitle

\begin{abstract}
	The main goal of this paper  is to prove   $L^1$-comparison and contraction  principles for weak solutions   of PDE system corresponding to a phase transition diffusion model of	Hele-Shaw type  with addition of a linear drift.
The  flow is considered with  a source term  and subject to mixed homogeneous boundary conditions : Dirichlet and Neumann.  The PDE   can be focused to model for instance  biological applications  including multi-species   diffusion-aggregation models and pedestrian dynamics with congestion.   Our approach  combines  DiPerna-Lions renormalization  type with Kruzhkov device of doubling and de-doubling variables.  The  $L^1$-contraction principle allows afterwards to handle the problem in a general framework of nonlinear semigroup theory in $L^1,$ taking thus advantage of this strong theory to study existence, uniqueness, comparison of weak solutions, $L^1$-stability  as well as many further questions. 
 \end{abstract}
\noindent  \textbf{Journal:} Mathematical Models and Methods in Applied Sciences, Vol. 33, No. 07, 1545-1576 (2023)\\
\url{https://doi.org/10.1142/S0218202523500355}  

\section{Introduction and preliminaries}
\setcounter{equation}{0}

\subsection{Introduction and main contributions}

Let $\Omega\subset\RR^N$ be a bounded   open set.     We are  interested in the existence and  uniqueness as well as the $L^1$-comparison principle for the weak solution  concerning the  PDE of the type
\begin{equation}\label{pdetype}
  \left\{
	\begin{array}{l}
		\displaystyle \frac{\partial u }{\partial t}  -\Delta p +\nabla \cdot (u  \: V)= f \\    \\
		\displaystyle u \in \sign(p)\end{array}\right.  \quad  \hbox{ in } Q:= (0,T)\times \Omega.
		\end{equation}
Here,   $\sign$ is the maximal monotone graph defined in $\RR$ by
$$\sign(r)= \left\{ \begin{array}{ll}
	\displaystyle 1 &\hbox{ for any }r>0\\  
	\displaystyle [-1,1]\quad & \hbox{ for } r=0\\
	\displaystyle -1 &\hbox{ for any }r<0, \end{array}\right. $$
$V$ and $f$ are the given  velocity field and source term respectively, satisfying besides assumptions we precise next.  In the case of nonnegative solution (one phase problem), the problem may be written in the widespread  form
\begin{equation} \label{pdetype+}
	\left\{
	\begin{array}{l}
		\displaystyle \frac{\partial u }{\partial t}  -\Delta p +\nabla \cdot (u  \: V)= f \\    \\
		\displaystyle  0\leq u\leq 1,\:  p\geq 0,\:  p(u-1)=0 \end{array}\right.  \quad  \hbox{ in } Q.
\end{equation}

The linear version of the problem which corresponds to the case where  $\sign$'s graph is replaced by the identity ; that is to take $p=u,$ the problem matches with Fokker-Planck equation.  Existence,  uniqueness and stability of weak solutions for this case   is studied in \cite{Figali} with possibly  linear degenerate diffusion, additive noise and $BV-$vector fields  $V$ in $\RR^N.$  One can see moreover the work  \cite{LebrisLions} where the case of  irregular coefficient  and the associated stochastic differential equation  are treated.    

Nonlinear versions of Fokker-Planck equation   cover  several   mathematical models  in biological applications and pedestrian flow.   The system \eqref{pdetype}  is a density dependent flow model  which may be obtained following   hydrodynamical approach  for collective motion.  Examples may be found in \cite{BellomoAl}. 
In the PDE system \eqref{pdetype}, the density constraint $\vert \rho\vert \leq 1$    is strongly connected to   the microscopic non-overlapping constraint between the agents, and its coupling in a second order equation with linear drift is natural in many settings.  
This becomes widespread in the  description of dynamics with congestion, particularly in pedestrian flow (cf. \cite{MRS1}) and in 
biological applications  including multi-species   diffusion-aggregation models (cf. \cite{Car} and the references therein). 
Indeed, even if the transport equation (linear or nonlinear) remains to be   the master equation for these type of phenomena,   one needs to go with second order terms   to perform some local behavior likely connected to  the  ‘thinking’ concept of the  agents (cf. \cite{Hugues}).    In general, the paths followed by the agents in   transport equation are chosen a priori  independently from their local dynamics.   This may lead to greater accumulation of agents obstructing the non overlapping constraint $\vert \rho\vert \leq 1$.  The introduction of the second variable $p$ with   the complementary condition $\rho\in \sign(p)$ typically allows to avoid this obstruction and describes  the  motion of congested zones.   Roughly speaking, while the drift  manages the dynamic  of the crowd with an overview   vector field  $V$, the second order term gives rise to patch the dynamic  in the congested zones  with a local view    looking out to    allowable  neighbors positions. 	A typical example of this point of view remains to be   the 	constrained diffusion-transport equation  which was performed   in the pioneering work  by B.  Maury and al.   (cf.  \cite{MRS1})  using a gradient flow in the Wasserstein space of probability measures.   In \cite{MRS1}, this coupling  was performed  for one population leading to one phase Hele-Shaw problem ; i.e. the   main actor for the dynamic remains to be the density of lonely population $0\leq \rho\leq 1$. However, the case of two-species  in inter-actions and occupying the same habitat,  like  diffusion-aggregation models,  may be described by    two-phase Hele-Shaw problem  of the type \eqref{pdetype},  where the unknown $\rho$ represents through its positive and negative parts  the densities of each specie respectively. The source term $f$ can model reaction phenomena connected  to agent  supply  in biological models. This happens in particular when one deals with reaction diffusion system coupling the equation    \eqref{pdetype} with other PDE. As to the boundary condition, we'll focus on homogeneous (null)  mixed one. Neumann boundary condition  is connected to the absence of crossing  boundary possibilities ; i.e. no mass goes through the boundary. And, Dirichlet one for $p$ is connected to the possibility of crossing other part of the boundary (exits) without any charge. One can see \cite{EIG} for other possibilities of boundary conditions and their interpretations. 

In general, the problem is of elliptic-hyperbolic-parabolic type. Despite the broad  results  on  this class of nonlinear PDE,  the structure of  \eqref{pdetype} and  \eqref{pdetype+}  excludes them definitely  out the scope  of the current literature, at least   concerning uniqueness (cf. \cite{AltLu,AnIg1,AnIg2,Ca,IgUr,Ot}, see also the expository paper \cite{AnIgSurvey} for a complete list of refs).  In spite of the ''hard'' non-linearity connected to the  sign graph, the    Fokker-Planck look  of the equation with linear drift seems to be very fruitful for the analysis based  on gradient flow in the space of probability measures equipped with Wasserstein distance (the distance arising in the Monge-Kantorovich optimal transport problem) in the one phase case  \eqref{pdetype+}  with  no-flux Neumann boundary condition and no-source ; i.e.  $f\equiv 0.$  Numerous results on the existence, properties and estimates on the weak solution were elaborated in the last decade  (one can see for instance \cite{MS} and the references therein).   Other interesting  progress  on qualitative properties of a   solution was obtained using the notion of viscosity  solution (one can see for instance \cite{AKY} and the references therein). The key utility of the viscosity concept   is the  skill to describe the pointwise behavior of the free boundary evolution.  Nonetheless, the uniqueness is still an open problem and seems to be   hard in general. In the case of monotone velocity field $V,$ the uniqueness of absolute continuous solution in the set of probability equipped with Wasserstein distance  has been obtained craftily in \cite{DM}  for \eqref{pdetype+}, in the framework of gradient flow in euclidean Wasserstein space  again in the case of Neumann boundary condition and $g\equiv 0$.

Our main results  in this  paper concern  the  $L^1$-comparison and contraction  principles   for the  diffusion-transport equation of the type \eqref{pdetype} subject to mixed homogeneous boundary conditions : Dirichlet and Neumann. This allows to handle  the problem in the context of nonlinear semigroup theory in $L^1(\Omega)$  to prove  thus  existence, uniqueness, comparison of weak solutions, $L^1$-stability  as well as many other questions. We'll not cover  all the possibilities with this framework here, but let us mention at least its prospect  to tackle the continuous dependence with respect to the $\sign$ graph, including for instance  the connection between this problem and the so called incompressible limit of the porous medium equation even in the singular case.    This subject would be likely treated  in forthcoming works. 
 We tackle the problem using a renormalization approach of  DiPerna-Lions type   (cf.   \cite{DiLions}) combined with Kruzhkov device of doubling variables.   We treat the problem in the case of outgoing vector field velocity on the boundary which remains practically useful for many applications like in crowd motion.  On the hand, this condition seems to be optimal in the case of Dirichlet boundary condition (see the counter-example  in Remark \ref{RcondV}).  
 
 Among the propose of this paper, let us mention the treatment of  the one phase problem \eqref{pdetype+}  which will be concerned as well  with sufficient condition on $V$ and $f.$  For the application in  crowd motion, the   conditions  may be  heuristically connected to the reaction of the concerned population at the position $x\in \Omega$ and time $t>0$ within   ''congested''  circumstance.    Performing its value  regarding to the divergence of the velocity vector field   $V$ may avoid the congestion.

 \subsection{Motivation and related works}
The nonlinear equation of the type \eqref{pdetype} is usually called Hele-Shaw Flow. Indeed, in the case where $V\equiv 0,$ the equation is a  free-boundary problem modeling the evolution of  a  slow  incompressible viscous  fluid  moving  between  slightly  separated plates  (cf. \cite{ElOk,Ri,SaTa,ElJa,DiFr,Vi} for physical and mathematical formulation).  The equation appears also as the viscous incompressible limit of  the porous medium equation (see for instance \cite{Igshaw,GQ,BeIgsign} and the references therein).

The study of the Hele-Shaw flow with a reaction term (without drift) goes back to \cite{BeIgsing} and \cite{BeIgNeumann} in the study of the  limit of the $m-$porous medium equation as $m\to\infty$. Concrete models with a dissimilar typical reaction term appeared after in the study of the tumor   growth (cf. \cite{PQV1}). The emergence of  a drift in the  Hele-Shaw flow appeared for the first time  in the study of congested crowd motion (cf. \cite{MRS1,MRS2,MRSV}).  In this case, the function $u$ is assumed  to be nonnegative to model  the density of  population and the constraints on $u$  prevents  an evolution beyond a given threshold.

Completed with boundary condition (Dirichlet, Neumann or mixed Dirichlet-Neumann) and initial data, the questions of existence and uniqueness of a weak solution is well understood by now in the case where $V\equiv 0$ (one can see for instance the papers \cite{BeIgsing,BeIgNeumann} for the case where $f=f(.,u)$, and  \cite{PQV1} for a different  structure of a reaction term). For general $V,$  though the existence of a solution seems to be more or less well understood, the uniqueness question  seems to be very delicate and challenging. The main difficulties  comes from the combination between the  ''hard'' non-linearity $\sign$ (so called hard-congestion) and the arbitrary pointing of  velocity field  for the drift.  Although the first order term is linear, one notices the hyperbolic character of the equation (outside the congested region).  This motivates the questions of existence and uniqueness of a weak solution,  and challenges new as well as  standard  techniques for first order hyperbolic  equations  and second order hyperbolic-parabolic structures as well. In particular, one sees that setting   $$\displaystyle
b=\left(I+{\sign}^{-1}\right)^{-1}\quad \hbox{ and }\quad  \beta=\left(I+\sign\right)^{-1},$$
$\displaystyle v:=u -p$,  the problem \eqref{pdetype} falls into the scope of    diffusion-transport equation of the type
 \begin{equation} \label{pdetypemod}
 \displaystyle \frac{\partial b(v) }{\partial t}  -\Delta \beta(v)  +\nabla \cdot ( V(x)\Phi(v) )= f \quad  \hbox{ in } Q:= (0,\infty)\times \Omega.
 \end{equation}
The study of this class of degenerate parabolic problems has under-gone a considerable progress in the last   twenty  years, thanks to the fundamental paper of J. Carrillo \cite{Ca}  in which the Kruzhkov device of doubling of variables was extended to this class of hyperbolic-parabolic-elliptic problems.  In \cite{Ca},  the appropriate notion of  solution was established : the so called weak-entropy solution. This definition (or, sometimes, parts of the uniqueness techniques of \cite{Ca}) led to many developments, we refer  to the expository paper \cite{AnIgSurvey} for a survey and complete references on this subject.  As  much as restrictive  it is, the approach of  \cite{Ca} serves out primarily and  ingeniously  the case where $V$ does not depend on space and  homogeneous Dirichlet boundary condition.  Strengthened with $L^1$-nonlinear semi-group techniques (cf. \cite{AnIgSurvey}),  the approach of  \cite{Ca}  enables  to  achieve successful as well the uniqueness of weak solution in the ''weakly degenerate''  convection-diffusion problems of parabolic-elliptic type ; that is  $[\beta=0] \subseteq  [\Phi = 0] $.  As far as  Dirichlet boundary condition is concerned (cf. \cite{Ca,IgUr,AnIg1,AnIg1,AnIgSurvey}), the case of homogeneous Neumann  was  treated in \cite{AnBo}.  Yet, the notion of entropic solution  \`a la Carillo is definitively the suitable notion for general  problem of the type \eqref{pdetypemod}. Nevertheless, following the theory of  hyperbolic  equation,  the case of linear drift is a particular case for which one expects the uniqueness of weak solution.   As   far  as  we  are  aware  of,   this  case remains to be out of the scope of the current literature on this subject.   

In spite of the serving of the second order term, we can not   pass over the forcefulness  of  the first order term, which is here linear and carrying  away along a vector field $V$ with an arbitrary pointing in $\Omega.$  Heuristically, in the likely case  $p\equiv 0,$  the PDE  turns into a linear initial-boundary value problem for continuity equation. In general, one guess the uniqueness of weak solution under reasonable assumption on $V$, like  bounded total variation coefficients.  As far as the regularity of $V$ is involved, the way how to handle the boundary condition is hooked to the proofs of  uniqueness and  adduce different difficulties to the problem.  Also, the treatment of boundary conditions is  a crucial step in the proof of uniqueness.  This is closely connected in some sense to the regularity of the solution as well as to the pointing direction of $V$ on the boundary. This  affects the manner to treat the weak/strong trace of the flux on the boundary (see the papers \cite{Ambrosio}, \cite{ACM}, \cite{Anzellotti}, \cite{ChenFrid}, \cite{ChenTZ} and also \cite{ACM}, \cite{AnIg2},  \cite{AnBo}). The approach based on the concept of renormalized solution introduced in 1989 by Lions and DiPerna \cite{DiLions} for tangential velocity field seems to be  unmissable  and powerful in general for the proof of   uniqueness of weak solution for this kind of problem.   This concept was extended to bounded domains in $\RR^N$ with inflow  boundary conditions and velocity field  with a kind of Sobolev regularity in \cite{Boyer},  \cite{BoyerFabrie} and \cite{Mischler}.  These results were generalized to velocity fields  with BV regularity in \cite{Ambrosio,AC,ACM,CDS}.  The concept was used also in  \cite{Figali} and \cite{LebrisLions} to tackle the existence and uniqueness of weak solutions for Fokker-Planck equation and its associated stochastic differential equation in some extreme cases like linear degenerate diffusion with irregular coefficient, additive noise and  $BV-$vector fields  $V$. 

In this paper, we show how to use this approach in the presence of second order term of Hele-Shaw type to prove the uniqueness of weak solution in the contest of mixed Dirichlet-Neuman boundary condition with an outward pointing    velocity fields $V$ on the boundary.  Otherwise, one needs   to be more careful with the treatment of  flux trace  on the boundary.  For instance, in the case of purely Dirichlet boundary condition we'll see that we can loose the uniqueness in the absence of outward pointing boundary condition on $V.$  

\subsection{Plan of the paper}
  In the following subsection, we give the main assumptions we'll use throughout the paper and the definition of weak solution we are dealing with.  
Section 2, is devoted to the $L^1$-comparison principle for the weak solution. We introduce  renormalized like formulations and use them with 
doubling and de-doubling variable techniques à la Kruzhkov inside $\Omega,$  to prove first that weak solutions satisfies some kind of local Kato's inequalities.    Then, linking the outgoing  assumption on the velocity field  $V$ on  the boundary with  the distance-to-boundary function, we prove that we can go with our Kato's inequalities up to the boundary proving thus  $L^1$-comparison principle, and then deduce the uniqueness. 
In section 3, we prove the existence of a weak solution by using nonlinear semigroup theory governed by $L^1$-accretive operator. Here the main ingredient is to use the $L^1$-contraction principle for weak solution of  stationary problem associated with the $\epsilon-$Euler implicit time discretization of the  evolution problem. Then, we pass to the limit in the so called $\eps-$approximate solution  and prove that the limit is the weak solution of the evolution problem. 
  Section 4  is  devoted to some remarks, comments  and possible  extensions. 
  At last, in Section 5, for completeness  we give a  complement for the proof of de-doubling  variables process to handle space dependent vector field with arbitrary pointing out.

 \subsection{Preliminaries, remarks and    main assumptions}\label{sectionprem}

 We assume that   $\Omega\subset\RR^N$ is a bounded  open set, with regular boundary  of class $\C^2,$ splitted into regular partition $\partial \Omega=\Gamma_D\cup \Gamma_N,$ such that $\Gamma_D\cap \Gamma_N=\emptyset$ and 
 $$\mathcal L^{N-1}(\Gamma_D)> 0.$$
 We consider $H^1_0(\Omega)$ (resp. $H^1_D(\Omega)$) the usual  space of functions in the Sobolev space $H^1(\Omega),$  with  null trace on the boundary $\partial\Omega $ (resp. $\Gamma_D).$    
  For any $h>0,$  we denote by 
\begin{equation}\label{xih}
	 \xi_h(t,x)=\frac{1}{h}\min\Big\{h,d (x,\partial \Omega)\Big\} \quad \hbox{ and } \quad \nu_h(x)=-\nabla \xi_h , \quad \hbox{ for any }x\in \Omega,
	\end{equation} 
where $d(.,\partial \Omega)$  names  the euclidean distance-to-the-boundary function.
 We see that    $    \xi_h \in H^1_0(\Omega) $ is    concave, $0\leq \xi_h\leq 1$ and
\begin{equation}\label{nuh}
	 \nu_h(x) =  - \frac{1}{h}\nabla\:  d(.,\partial \Omega) ,\quad \hbox{ for any }x\in \Omega \hbox{ s.t. }d(x,\partial \Omega)<h\leq h_0 \hbox{ (small enough)}. \end{equation} 
 In particular, for such $x,$ we have  $ h\nu_h(x) = \nu(\pi(x)),$ where $\pi(x)$  design the projection of $x$ on the boundary $\partial \Omega,$  and  $ \nu(y) $   represents the   outward unitary normal to  the boundary $\partial \Omega$ at $y.$  

\medskip  
 We denote by $\signp$   the maximal monotone graph given by  
$$\signp(r)= \left\{ \begin{array}{ll}
	\displaystyle 1 &\hbox{ for  }r>0\\  
	\displaystyle [0,1]\quad & \hbox{ for } r=0\\
	\displaystyle 0 &\hbox{ for  }r<0. \end{array}\right. $$ 
Moreover, we define $\sign_0$ and  $\sign^{\pm}_0,$ the   discontinuous   applications defined from $\RR$ to $\RR$ by  
  $$\sign_0(r)= \left\{ \begin{array}{ll}
  	\displaystyle 1 &\hbox{ for  }r>0\\  
  	\displaystyle 0\quad & \hbox{ for } r=0\\
  	\displaystyle -1 &\hbox{ for  }r<0,  \end{array}\right.,\quad \spo(r)= \left\{ \begin{array}{ll}
  	\displaystyle 1 &\hbox{ for  }r>0\\   
  	\displaystyle 0 &\hbox{ for  }r\leq 0 \end{array}\right.    \quad   \smo(r)= \left\{ \begin{array}{ll}
  	\displaystyle 0 &\hbox{ for  }r\geq 0\\   
  	\displaystyle 1 &\hbox{ for  }r< 0 \end{array}\right.   $$

 \bigskip 
Throughout the paper, we  assume that $u_0$ and the velocity vector filed $V$ satisfy  the following assumptions :  
 \begin{itemize}
 	\item  $\displaystyle  u _0\in L^\infty(\Omega)$  and $0\leq \vert u_0\vert \leq 1$ a.e. in $\Omega.$
 	\item $\displaystyle V \in  W^{1,2}(\Omega)^N$, $\nabla \cdot V\in L^\infty(\Omega) $  and satisfies (outward pointing  velocity vector field condition on the boundary)
 	\begin{equation} \label{HypV0}
 		V\cdot \nu \geq 0\quad \hbox{ on }\Gamma_D  \quad \hbox{ and } 	\quad V\cdot \nu = 0\quad \hbox{ on }\Gamma_N .
 		\end{equation}
 \end{itemize}  

 \begin{remark}
 \begin{enumerate}
 	\item 
 Since $\displaystyle V \in  W^{1,2}(\Omega)^N$ and  $\nabla \cdot V\in L^\infty(\Omega) ,$   then $V\cdot \nu \in H^{-\frac{1}{2}}(\partial\Omega).$  So, a priori  \eqref{HypV0}  needs to be understood in a weak sense ; i.e. 
 \begin{equation}\label{HypVGammaN0}
 	\int_\Omega V\cdot \nabla \xi \: dx +\int_\Omega \nabla \cdot V\: \xi\: dx \geq 0,\quad \hbox{ for any }0\leq \xi\in H^{1}(\Omega)
 \end{equation}
and 
 \begin{equation}\label{HypVGammaN}
	\int_\Omega V\cdot \nabla \xi \: dx +\int_\Omega \nabla \cdot V\: \xi\: dx = 0,\quad \hbox{ for any }0\leq \xi\in H^1_D(\Omega).
\end{equation}
  In particular,  this condition  implies that  
	\begin{equation}\label{HypV}
	\liminf_{h\to 0}  \frac{1}{h} \int_{D_h}  \xi\:   V(x)   \cdot \nu(\pi(x))  \: dx \geq   0 , \quad \hbox{ for any }0\leq \xi \in H^{1}(\Omega),
\end{equation}
 	where, for any $0<h\leq  h_0,$  $D_h$ is a  neighbourhood  of $\partial \Omega.$ Nevertheless,  in order to prove uniqueness we'll assume that \eqref{HypV} is fulfilled for any $0\leq \xi\in L^\infty(\Omega)$ (see Lemma \ref{pintegral}). We do not know if this is a consequence of  the assumption \eqref{HypV0} even if it remains be to true for a large class of  practical situations (see Remark \ref{RemboundaryCond1} and Remark  \ref{RemboundaryCond2}).

	\item   	 In some models of congested  crowd motion, the vector field velocity may be  given by 
	\begin{equation}\label{HypVexple}
		V=-\nabla d(.,\Gamma_D) ,\quad \hbox{ in }\Omega.
	\end{equation}
	This means that the pedestrian choose the euclidean geodesic trajectory to the escape door  to get away from the environment  $\Omega.$  See that in this case the velocity vector field is outward pointing  on $\Gamma_D$ as we need,  but strictly inward pointing  on $\Gamma_N.$  To fall into the scope of \eqref{HypV0} one needs for instance to consider  \eqref{HypVexple} outside a small neighbourhood of $\Gamma_N$ and set $V$ to be tangential in this neighbor. This example constitute a  typical practical  situation (among many others)  for which all the results of this paper may be applied.  By the way, one   sees that in this case  \eqref{HypV} is fulfilled for any $0\leq \xi\in L^\infty(\Omega)$  (thanks to   \eqref{nuh})
\end{enumerate}

\end{remark}

\bigskip 
To begin with, we consider first the problem
\begin{equation}
	\label{cmef}
	\left\{  \begin{array}{ll}\left.
		\begin{array}{l}
			\displaystyle \frac{\partial u }{\partial t}  -\Delta p +\nabla \cdot (u  \: V)= f \\
			\displaystyle u \in \sign(p)\end{array}\right\}
		\quad  & \hbox{ in } Q \\  \\
		\displaystyle p= 0  & \hbox{ on }\Sigma_D:= (0,T)\times \Gamma_D\\  \\
		\displaystyle (\nabla p- u  \: V)\cdot \nu = 0  & \hbox{ on }\Sigma_N:= (0,T)\times \Gamma_N\\  \\
		\displaystyle  u (0)=u _0 &\hbox{ in }\Omega,\end{array} \right.
\end{equation}
where $f\in L^1(Q)$ is given.

\begin{definition}[Notion of solution] \label{defws} A  couple $(u ,p) $  is said to be a weak solution of   \eqref{cmef}
	if $(u ,p)\in  L^\infty(Q) \times  L^2
	\left(0,T;H^1_D(\Omega)\right)$, $\displaystyle    u \in \sign (p)$ a.e. in $Q$  and
	\begin{equation}
		\label{evolwf}
		\displaystyle \frac{d}{dt}\int_\Omega u \:\xi+\int_\Omega ( \nabla p -  u \:V) \cdot  \nabla\xi   =     \int_\Omega f\: \xi  , \quad \hbox{ in }{\D}'(0,T),\quad \forall \: 	\xi\in H^1_D(\Omega).
	\end{equation}
	We'll say plainly  that $u$ is a solution of \eqref{cmef} if $u\in \C([0,T),L^1(\Omega))$, $u(0)=u_0$ and there exists $p\in L^2
	\left(0,T;H^1_D(\Omega)\right)$ such that the couple $(u,p)$ is a weak solution of   \eqref{cmef}.
\end{definition}

 \begin{remark} \label{RcondV}
 The assumptions we suppose on $V$ look alike to be stronger regard  to the literature on (linear) continuity equation. We think it is possible to extend our  results for weak solution to more general $V.$  But,  in  the presence of the second order term, the problem  seems to be  heavy and much more technical. Let us give here at least a  short reminder on the  uniqueness of weak solution for  continuity equation, which could comprise some possible extensions (open questions !) for    the Hele-Shaw flow  with a  linear Drift. 
 
 	\begin{enumerate}
 
\item  The analysis  of the continuity equation  (without  a second order term)  in the case when $V$  has low regularity has   drawn considerable attention. For an overview of some of the main contributions, we refer to the lecture notes by Ambrosio and Crippa \cite{AC} (see also \cite{CDS} for more references on this subject). 
Counter-example for the uniqueness when $V$ is not enough regular (up to the boundary)  may be found in \cite{CDS}. Indeed, regardless the orientation of $V$  at the boundary, uniqueness may be violated as soon as $V$   enjoys BV regularity in every open set $\omega \subset\!\subset \Omega$, but not at the boundary of $\Omega$ (cf.   \cite{CDS} for more details and discussions on this subject). 
 
 \item   The outward pointing condition on $\Gamma_D$ seems to be an optimal  condition for uniqueness when one deals with purely Dirichlet boundary condition. Indeed,   one sees easily that the $1-$dimensional example  in $(0,2)$ : 
 \begin{equation} \label{explen1}
 	\left\{  \begin{array}{ll}\left.
 		\begin{array}{l}
 			\displaystyle \partial_t u -\partial_{xx} p +\partial_x  u = 0 \\
 			\displaystyle u \in \sign(p)\end{array}\right\}
 		\quad  & \hbox{ in } (0,\infty)\times (0,2) \\  \\
 		\displaystyle p(0)=p(2)=0  &    \hbox{ in } (0,\infty) \\\  \\ 
 		\displaystyle  u (0)\equiv 1 &\hbox{ in }(0,2),\end{array} \right.
 \end{equation}
 where $V\equiv 1$ is inward pointing on $0$ and outward pointing on $2,$ has many solutions.  Indeed, one can take any smooth function $F\: :\:  \RR\to \RR^+$ such that $F(r)=1$ for any $r\leq 0$ and $F(r)=0$ for any $r\geq 1,$ the function  $u(t,x)=F(t-x)$ for any $(t,x)\in [0,\infty)\times [0,2]$ is a solution of \eqref{explen1}.  Here $p\equiv 0$ in  $[0,\infty)\times [0,2],$ and the structure of Dirichlet boundary condition in \eqref{cmef}  leaves $u$ free  on the boundary. To handle/avoid  this kind of situation one needs to change the selection criteria on the boundary. 
 
\item When  the vector filed is outward pointing,  heuristically  the solution would not be worst at/near the boundary. 
Indeed, at least in the smooth case the solution may be simply “carried out” of the domain along the characteristics and, consequently, the behavior of the solution inside the domain is not substantially affected by what happens close to the boundary.  However, let us remind the   reader that examples are given in  \cite{CDS} discussing that even if $V$  is outward pointing at $\partial \Omega ,$ then uniqueness may be violated as soon as the regularity ($BV$ regularity) deteriorates at the boundary.


	\end{enumerate}
 \end{remark}

\section{$L^1-$Comparison principle}
\setcounter{equation}{0}
In this section, we focus first on the  uniqueness and $L^1-$Comparison principle of weak solution. Our first main result is the following. 
 \begin{theorem} \label{compcmef}
Under the assumptions of Section \ref{sectionprem}, we assume moreover that 
	\begin{equation}\label{HypVstg}
	\liminf_{h\to 0}  \frac{1}{h}  \int_{\Omega\setminus \Omega_h}  \xi\:   V(x)   \cdot \nu(\pi(x))  \: dx \geq   0 , \quad \hbox{ for any }0\leq \xi \in L^2(\Omega).
\end{equation}
If $(u_1,p_1)$ and $(u_2,p_2)$ are two weak solutions of \eqref{cmef} associated with $f_1,\ f_2\: \in L^1(Q)$ respectively, then there exists $\kappa \in L^\infty(Q),$ such that $\kappa\in \signp(u_1-u_2)$ a.e. in $Q$ and 
	\begin{equation}
		\label{evolineqcomp}
	\frac{d}{dt}	\int_\Omega ( u _1-u _2)^+ \: dx \leq \int_\Omega \kappa\:     ( f_1-f_2)\: dx ,\quad \hbox{ in }\D'(0,T).
	\end{equation}
	In particular, we have
	$$ 	\frac{d}{dt} \Vert u_1-u_2\Vert_{1} \leq \Vert f_1-f_2\Vert _{1},\quad \hbox{ in }\D'(0,T).$$
	Moreover, if $f_1\leq f_2,$  a.e. in  $Q,$ and $u_1,u_2$ are two corresponding solutions  satisfying $u_1(0)\leq u_2(0) $ a.e. in $\Omega,$   then
	$$u_1\leq u_2,\quad \hbox{ a.e. in  }Q.$$
	\end{theorem}

\begin{remark}\label{RemboundaryCond1}
 
 The assumption \eqref{HypVstg}	 is technical for the the proof of Theorem \ref{compcmef}. It is fulfilled for a large  class of vector field $V$, like for instance the case  $V$ is  outward pointing in the neighborhood of  the boundary. We postpone the technicality of this assumption in  Remark \ref{RemboundaryCond1} after the proof of Theorem \ref{compcmef}. 
 
\end{remark}

As an immediate  and primary consequence of Theorem \ref{compcmef}, we have the following results. Further consequences, will be given in the following sections.   
\begin{corollary}\label{Cuniq}
Let  $u_0\in L^\infty(\Omega)$ and $f\in L^1(Q).$  

\begin{enumerate}

\item There exists at most one  $u\in L^\infty(Q)$ such that, there exists $p\in L^2
\left(0,T;H^1_D(\Omega)\right)$, $\displaystyle    u \in \sign (p)$ a.e. in $Q$  and the couple $(u ,p)$ satisfies 
\begin{equation}
	\displaystyle   - \int_Qu\: \xi  \:\partial_t \psi \: dxdt +\int_Q ( \nabla p -  u \:V) \cdot  \nabla\xi \: \psi\: dxdt   =     \int_Q f\: \xi \psi\: dxdt + \int_\Omega u_0\: \psi(0)\:\xi \: dx  , 
\end{equation} 
for any $	\psi\in \D([0,T))\hbox{ and }\xi\in H^1_D(\Omega).$

\item 	The problem \eqref{cmef} has at most one solution, in the sense that $u\in \C([0,T),L^1(\Omega))$, $u(0)=u_0$ and there exists $p\in L^2
\left(0,T;H^1_D(\Omega)\right)$ such that the couple $(u,p)$ satisfies \eqref{cmef}.

\item If $u_0\geq 0$ and $f\geq 0,$ then any   solution is nonegative. 
	\end{enumerate}

\end{corollary}

 \begin{remark}
 	The main difference between part $1$ and part $2$ in Corollary \ref{Cuniq} concerns the way one handle  initial data $u_0.$ While the  first assertion handle it in a weak sense, the second one is doing the job in a strong sense (since $u$ is assumed there  to be in $\C([0,T),L^1(\Omega))$). Of course the second notion implies  the first one. In this paper, we favor the  notion of solution in $\C([0,T),L^1(\Omega)),$  since this is automatically  obtained through nonlinear semigroup theory. After all,   one  sees that both solutions will exist and are  equivalent. 
 	   
 \end{remark}

The main tool to prove this result is doubling and de-doubling variables. Since the degeneracy of the problem,   we prove first that   a weak solution satisfies some kind of (new) renormalized formulation, and then  carry out  trickily  doubling and de-doubling variables device to prove \eqref{evolineqcomp}.

  Recall that,   the merely transport equation case corresponds to the situation where  $p\equiv 0.$ In this case,    the renormalized formulation reads
	\begin{equation}  	\partial_t \beta (u)  + 	V\cdot \nabla \beta(u)  +u\:  \nabla \cdot V   	\beta'(u)
		=   f \:   \beta'(u)
		\quad \hbox{ in }  \D'(Q)  ,	\end{equation}
for any $\beta\in \mathcal C^1(\RR)$, where from now on, for any $z\in L^1(\Omega)$  the notation $V\cdot \nabla z$ needs to be understood in the sense of distributions   as follows \begin{equation}\label{formdiv}
 V\cdot \nabla z = \nabla  \cdot  (z\: V)- z\: \nabla \cdot V\:  .
\end{equation}

%

\begin{proposition} \label{prenormal}
Under the assumptions of Section \ref{sectionprem}, if $(u,p)$ is a weak solution of \eqref{cmef},  then
 	\begin{equation} \label{renormal+}	\begin{array}{c}
		\partial_t \beta (u)  - \Delta p^+  + 	V\cdot \nabla \beta(u)  +u\:  \nabla \cdot V   ( 	\beta'(u) 	\: \chi_{[p=0]}   +    \spo(p)  )      \\  \\
		\leq   f \:  (\beta'(u) \:  \chi_{[p=0]}  + \spo(p) )
	  \quad \hbox{ in }  \D'(Q)
\end{array}  	\end{equation}
and
 \begin{equation} \label{renormal-}
 	\begin{array}{c}
		\partial_t \beta (u)  - \Delta p^-   + 	V\cdot \nabla \beta(u)  +  u\: \nabla \cdot V   ( 	\beta'(u)\:
	\:  \chi_{[p=0]}  -    \smo(p) )       \\  \\
		\leq   f \:  ( \beta'(u)\:   \chi_{[p=0]}  -  \smo(p)  )
		\quad \hbox{ in }  \D'(Q),
	\end{array}
	\end{equation}
	for any $\beta\in \mathcal C^1(\RR)$ such that $0\leq \beta'\leq 1.$
\end{proposition}

\begin{remark}
 	The formulations \eqref{renormal+} and  \eqref{renormal-} describe  some kind of renormalized  formulation     for the solution  $u.$   For the  one phase Hele-Shaw problem this formulation reads simply as follows
	\begin{equation} \label{renormalHS+}	\begin{array}{c}
			\partial_t \beta (u)   - \Delta p  + 	V\cdot \nabla \beta(u)  +  u\:  	\beta'(u)\:  \nabla \cdot V \:
			\chi_{[p=0]}   +   \nabla \cdot V   \chi_{[p\neq 0]}    \\  \\
			\leq   f \:  (\beta'(u)  \:  \chi_{[p=0]}  + \chi_{[p\neq 0]})
			\quad \hbox{ in }  \D'(Q),
		\end{array}
	\end{equation}
	for any $\beta\in \mathcal C^1(\RR)$ such that $0\leq \beta'\leq 1.$ 
\end{remark}

We prove  this results in two steps,   we begin with the case where $\beta \equiv 0$. Then  
proceeding skillfully  with the positive and negative part of $p$, and using a smoothing  procedure with  a commutator   à la  DiPerna-Lions (cf. \cite{DiLions}, see also \cite{Ambrosio}), we  prove that any weak solution $u$  satisfies    renormalized   formulations \eqref{renormal+} and \eqref{renormal-}.

\begin{lemma}\label{leqp+}  	 
Under the assumptions of Section \ref{sectionprem},  if $(u,p) $ is a weak solution of \eqref{cmef}, then  the renomalized formulations \eqref{renormal+} and \eqref{renormal-} are fulfilled with  $\beta \equiv 0$ in  $\D'((0,T)\times \overline \Omega).$ That is 
		\begin{equation}\label{eqp+1}
		-\Delta p^+ +(\nabla \cdot V-f)\spo(p)  \leq 0 \quad \hbox{ in }\D'((0,T)\times \overline \Omega),
	\end{equation}
and 	\begin{equation}\label{eqp-1}
		-\Delta p^- +(\nabla \cdot V+f)\smo(p)\leq 0   \quad \hbox{ in }\D'((0,T)\times \overline \Omega).
	\end{equation}
	Moreover, we have 
	\begin{equation}\label{eqp+2}
		\partial_t u    - \Delta p^+  +\nabla \cdot(  u\:  V )       +
		\nabla \cdot  V\:   \smo(p)     \leq    f\:  (1-  \smo (p) )    \quad \hbox{ in }\D'(Q),
	\end{equation}
 \end{lemma}
	\begin{proof}  By density, it is enough to prove  \eqref{eqp+1}, \eqref{eqp-1} and \eqref{eqp+2} with any test function of the type $\psi\xi,$ with non-negative $\psi\in \D(0,T)$ and $\xi \in H^1(\Omega).$   To this aim, 
  we extend $p$ onto $\displaystyle \RR\times\Omega$ by
		$0$ for any   $\displaystyle t\not\in (0,T),$ and for any $h>0,$  we  consider
		$$\displaystyle \Phi^h (t,x)=    \xi(x)\:  \psi(t)\:  \frac{1}{h}\int_t^{t+h}  \hepsp (p(s,x))\:  ds  ,\quad \hbox{ for a.e. }(t,x)\in Q,$$
		where  $\psi$ is extended  in turn  onto $\displaystyle \RR$ by
		$0,$   and $\hepsp$ is given by
\begin{equation} \label{hepsp}
	\hepsp(r)=  \min \left( \frac{r^+}{\eps}, 1 \right),\quad \hbox{ for any }r\in \RR,
\end{equation}
for arbitrary  $\eps >0.$  It is clear that  $\Phi_h \in W^{1,2} \Big( 0,T;H^1_D(\Omega) \Big)\cap L^\infty(Q)$    is an admissible test function for the weak formulation, so that 
		\begin{equation}\label{evolh0}
			-	\int\!\!\!\int _Q   u  \: \partial_t   \Phi^h 	 +	\int\!\!\!\int _Q (\nabla  p - V\:  u   ) \cdot \nabla \Phi^h \\  \\ =  	\int\!\!\!\int _Q f\:  \Phi^h.
		\end{equation}
See that 
		\begin{equation} \label{rel0} \int\!\!\!\int _Q   u  \: \partial_t   \Phi^h   = 	\int\!\!\!\int _Q   u \: \partial_t   \psi   \:  \frac{1}{h}\int_t^{t+h}   \hepsp (p((s))\:  ds +	\int\!\!\!\int _Q
			u(t)\:	 \frac{  \hepsp(p (t+h))- \hepsp(p (t))}{h} \: \psi(t)\:  \xi.
		\end{equation}
	Moreover,  using the fact that for a.e. $t\in (0,T),$ $-1\leq u(t)\leq 1$,  ${\hepsp}\geq 0$ and $ \hepsp(0)=0,$  we have
	$u(t,x)\: \hepsp(p(t,x))=  \hepsp(p(t,x))  $ and 	$  u(t,x)\: \hepsp(p(t+h,x)) \leq     \hepsp(p(t+h,x))  $  a.e.   $(t,x)\in Q.$  So, for any $h>0$ (small enough), we have 
	 	\begin{eqnarray}
			\int\!\!\!\int _Q
			u(t)\:	 \frac{  \hepsp(p (t+h))- \hepsp(p (t))}{h} \: \psi(t)\:  \xi
			&\leq &  	\int\!\!\!\int _Q
			\frac{    \hepsp(p (t+h))   -    \hepsp(p (t))  }{h}\: \psi(t)   \: \xi  \\
			&\leq &  	\int\!\!\!\int _Q
			\frac{   \psi (t-h) -   \psi (t) }{h}\:    \hepsp(p (t))     \: \xi  .
		\end{eqnarray}
		This implies that
		\begin{equation}
			\limsup_{h\to 0 }	\int\!\!\!\int _Q
			u(t)\:	 \frac{   \hepsp(p (t+h))-  \hepsp(p (t))}{h} \: \psi(t)\:  \xi  \leq   -	\int\!\!\!\int _Q  \partial_t   \psi \:     \hepsp(p (t))   \: \xi   ,
		\end{equation}
so that, by  letting $h\to 0$ in \eqref{rel0}, we get
$$	\lim_{h\to 0 }   \int\!\!\!\int _Q   u  \: \partial_t   \Phi^h   
\leq  0.$$ 
Then, by  letting $h\to 0$ in  \eqref{evolh0}  
 	\begin{equation}\label{ajout0}
 \int\!\!\!\int _Q (\nabla  p - V\:  u   ) \cdot \nabla ( \hepsp(p (t))\:\xi)    - 	\int\!\!\!\int _Q f\:  \hepsp(p (t))\: \xi  \leq \int\!\!\!\int _Q f\:  \hepsp(p (t))\: \xi. 	\end{equation}
On the other hand,   using again the fact  that  $u  \hepsp(p)=  \hepsp(p) $, a.e. in $Q,$    we have
		\begin{equation}
		\begin{array}{ll}
				\int\!\!\!\int _Q
			( \nabla  p-u \: V ) \cdot \nabla
			\Big(  \hepsp(p)\: \xi \Big)   \:  \psi 
			  & = 			  	\int\!\!\!\int _Q     \hepsp(p)   \:  \nabla p\cdot  \nabla  \xi  \:  \psi +  \int\!\!\!\int _Q   \vert \nabla p\vert^2\:( {\hepsp})'(p)\: \xi  \:  \psi   \\          
			& \hspace*{2cm}- \int\!\!\!\int _Q   V\cdot \nabla  (\xi   \hepsp(p)   )  \:  \psi   \\  
			  & \geq  
			\int\!\!\!\int _Q     \hepsp(p)   \:  \nabla p\cdot  \nabla  \xi  \:  \psi    + \int\!\!\!\int _Q \nabla \cdot V  \:    (\xi   \hepsp(p)   )  \:  \psi    , \end{array} 
		  \end{equation}
	  where we use \eqref{HypVGammaN} and the fact that   $\vert \nabla p\vert^2 \: ({\hepsp})'(p)  \geq 0.$ 
Thanks to \eqref{ajout0},  this implies that 
\begin{equation} 
	\int\!\!\!\int _Q   \nabla p\cdot \nabla \xi \: {\hepsp}(p)\: \xi  \:  \psi +  \int\!\!\!\int _Q  \nabla \cdot  V \:  \xi   \hepsp(p)    \:  \psi 
	\leq \int\!\!\!\int _Q f\:  \hepsp(p (t))\: \xi. 
\end{equation}   
Letting now	$\eps\to 0$, we get   	  \eqref{eqp+1}.    As to  \eqref{eqp-1}, it follows by  using the  that   $(-u,-p)$ is also a solution 
	 of  \eqref{cmef} with $f$ replaced by $-f,$  and   applying    \eqref{eqp+1} to $(-u,-p)$.      At last, recall that
\begin{equation} \label{veq0}
	\partial_t u-\Delta  p   +\nabla \cdot(  u\:  V )    =  f\: \quad \hbox{ in }  \D'(Q)   .
\end{equation}
So, summing   \eqref{eqp-1} restrained to $\D'(Q)$ and \eqref{veq0}, we get    \eqref{eqp+2}.
 \end{proof}

 \bigskip 
Now, in order to prove the proposition by using \eqref{eqp+1} and  \eqref{eqp+2}, we prove the following technical lemma.  
\begin{lemma} \label{lrenormal}
Let $u\in L^1_{loc}(Q),$  $F\in L^1_{loc}(Q)^N$ and $J_1 \in L^1_{loc}(Q)$ be such that
	\begin{equation}\label{form1}
		\partial_t u +	V\cdot \nabla u -\nabla \cdot F \leq J_{1} \quad \hbox{ in } \mathcal D'(Q)
	\end{equation}
	where $V\cdot \nabla u$ is taken in the sense $V\cdot \nabla u = \nabla \cdot (u \:V) - u \: \nabla \cdot V, $   in $\D'(Q).$ If
	\begin{equation}\label{form2}
		-\nabla \cdot F \leq J_{2} \quad \hbox{ in } \mathcal D'(Q),
	\end{equation}
for some $J_2\in L^1_{loc}(Q),$ 	then
	\begin{equation} \label{techineq}
		\partial_t \beta (u)+ 	V\cdot \nabla \beta(u) -\nabla \cdot F \leq {J_{1}}\beta'(u) +  {J_{2}}(1-\beta'(u))         \quad \hbox{ in }  \D'(\Omega),
	\end{equation}
	for any $\beta\in \mathcal C^1(\RR)$ such that $0\leq \beta'\leq 1.$
 \end{lemma}
\begin{proof}    We set $Q_\eps   := \left\{ (t,x)\in Q\: :\:  d((t,x),\partial Q )> \eps \right\}.$  Moreover, for any $z\in L^1_{loc}(Q),$ we denote by $z_\eps$ the usual regularization of $z$ by convolution given and denoted by 
$$z_\eps := z\star \rho_\eps, \quad \hbox{ in }Q_\eps, $$ where  $\rho_\eps $ is the usual  mollifiers sequence defined here in $ \RR\times \RR^N.$   
   It is not difficult to see that  \eqref{form1} and \eqref{form2}    implies respectively
	\begin{equation}\label{formeps1}
		\partial_t u_\eps + 	V\cdot \nabla u_\eps -\nabla \cdot F_\eps \leq {J_{1}}_\eps+ \C_\eps \quad \hbox{ in }  Q_\eps
	\end{equation}
and 
\begin{equation}\label{formeps2}
	-\nabla \cdot F_\eps  \leq  {J_{2}}_\eps\quad \hbox{ in }  Q_\eps, 
\end{equation}
	 where  $\C_\eps$ is the usual commutator given by
	$$ \C_\eps:= V\cdot \nabla u_\eps - (V\cdot \nabla u)_\eps.$$  Using \eqref{formdiv}, here $(V\cdot \nabla u)_\eps$ needs to be understood in the sense 
\begin{equation}
(V\cdot \nabla u)_\eps = (u\: V)\star  \nabla \rho_\eps - (u\: \nabla \cdot V)\star \rho_\eps,\quad \hbox{ in }Q_\eps.
\end{equation} 	
Multiplying  \eqref{formeps1} by $\beta'(u_\eps)$ and  \eqref{formeps2} by $1-\beta'(u_\eps)$ and adding the resulting equations, we obtain
	$$\beta'(u_\eps) \:     \partial_t u_\eps + \beta'(u_\eps) \:  	V\cdot \nabla u_\eps -\nabla \cdot F_\eps       \leq   \C_\eps\:  \beta'(u_\eps) +  {J_{1}}_\eps\beta'(u_\eps) +  {J_{2}}_\eps(1-\beta'(u_\eps) ) \quad \hbox{ in }  Q_\eps$$
	and then
	\begin{equation}\label{formeps}
		\partial_t \beta(u_\eps)  +	V\cdot \nabla \beta(u_\eps) -\nabla \cdot F_\eps \leq \C_\eps\:  \beta'(u_\eps)  + {J_{1}}_\eps\beta'(u_\eps) +  {J_{2}}_\eps(1-\beta'(u_\eps)  )      \quad \hbox{ in }  Q_\eps.
	\end{equation}
	Since $V\in W^{1,1}_{loc} (\Omega)$ and $\nabla \cdot V\in L^\infty(\Omega),$ it is well  known by now that, by taking a subsequence if necessary,  the  commutator converges to $0$ in  $L^1_{loc}(Q),$ as $\eps \to 0$   (see for instance Ambrosio \cite{Ambrosio}).  So, letting $\eps\to  0$ in \eqref{formeps}, we get  \eqref{techineq}.
\end{proof}

 \bigskip 
Now, let us  prove how to  get the renormalized formulations \eqref{renormal+} and \eqref{renormal-}.
\begin{proof}[Proof of Proposition \ref{prenormal}]
	Thanks to Lemma \ref{leqp+}, by using the fact that
	$$	\nabla \cdot  V\:   \smo(p)= -u\: 	\nabla \cdot  V\:   \smo(p) \hbox{ and } 	\nabla \cdot  V\:   \spo(p)= u\: 	\nabla \cdot  V\:   \spo(p)  ,$$ we see that    \eqref{form1}  and \eqref{form2}  are  fulfilled with
	$$ F:= \nabla p^+,   \quad  J_{1}:= (f- u\:  \nabla \cdot V )(1- \smo(p)) $$
	and  	$$ J_{2}:= ( f-u\nabla \cdot V )\: \spo(p)  . $$
	Thanks to Lemma \ref{lrenormal}, 	for any $\beta\in \mathcal C^1(\RR)$ such that $0\leq \beta'\leq 1,$  we deduce that
		\begin{equation} 	\begin{array}{c}
		\partial_t \beta (u)- \Delta p^+ + 	V\cdot \nabla \beta(u)  + ( u\:  \nabla \cdot V-f )(1- \smo(p))  \beta'(u)  \\  \\  +   ( f- u \nabla \cdot V  )\: \spo(p) ( \beta'(u)-1)       \leq  0\quad \hbox{ in }  \D'(Q),
	\end{array}  	\end{equation}
Using again the fact that  $\nabla \cdot V \: \spo(p)	=u\: \nabla \cdot V \: \spo(p)$ this  implies that
		\begin{equation} 	\begin{array}{c}
			\partial_t \beta (u)- \Delta p^+ + 	V\cdot \nabla \beta(u)  +  u\:  \nabla \cdot V\: (1- \smo(p))  \beta'(u)    +   u\nabla \cdot V \: \spo(p) (1-\beta'(u))    \\  \\     \leq   f  (  \spo(p) (1-\beta'(u)) + (1- \smo(p))  \beta'(u) )  \quad \hbox{ in }  \D'(Q).
	\end{array}  	\end{equation}
	 
That  is
	\begin{equation} 	\begin{array}{c}
		\partial_t \beta (u)- \Delta p^+ + 	V\cdot \nabla \beta(u)  +  u\:  \nabla \cdot V
	( \beta'(u)( 1-\smo(p) - \spo(p) )   +  \spo(p) )     \\  \\     \leq   f  (     \beta'(u)   (1   -  \smo(p) -\spo(p) )    + \spo(p)  )  \quad \hbox{ in }  \D'(Q),
\end{array}  	\end{equation}
and then 
	\begin{equation} 	\begin{array}{c}
		\partial_t \beta (u)- \Delta p^+ + 	V\cdot \nabla \beta(u)  +  u\:  \nabla \cdot V
		( \beta'(u) \chi_{[p=0]}     +  \spo(p) )     \\  \\     \leq   f  (     \beta'(u)  \chi_{[p=0]}   + \spo(p)  )  \quad \hbox{ in }  \D'(Q).
\end{array}  	\end{equation}
Thus \eqref{renormal+}. At last, using the fact that  the couple $(-u,-p)$ is a solution of  \eqref{cmef} with $f$ replaced by $-f,$ and applying  \eqref{renormal+} to $(-u,-p)$, we deduce  \eqref{renormal-}. 

\end{proof}

See here that   the attendance of a nonlinear second order term creates an  obstruction to use standard approaches  for the uniqueness of weak solutions  based effectively on  the linearity of the first order term.  In order to   prove uniqueness in our case, we proceed by  doubling and de-doubling variable techniques. To this aim, we'll use the fact that the renormalized formulation implies some kind of  local  entropic inequalities.  This is the aim of the following   Lemma.

\begin{lemma}\label{pentropic}
Under the assumptions of Section \ref{sectionprem}, 	if   $(u,p)$ is a weak solution of  \eqref{cmef}, then we have
		\begin{equation} \label{evolentropic+}
		\begin{array}{c}
			\partial_t   (u-k)^+   - \Delta p^+  + 	\nabla \cdot(  (u-k)^+ \:  V)  +k\:  \nabla \cdot V \: \spo (u-k)       \\  \\
			\leq   f \:  \spo(u-k)
			\quad \hbox{ in }  \D'(Q),\quad \hbox{ for any }k<1
	\end{array}  	\end{equation}
	and
	\begin{equation} \label{evolentropic-}
		\begin{array}{c}
			\partial_t   (u-k)^-	-\Delta   p^-	+	 \nabla \cdot(  (k-u)^+  V)  -k  \:  \nabla \cdot V \:    \spo(k-u)      \\  \\       \leq  -f \:    \spo(k-u)   \quad \hbox{ in }  \D'(Q)\quad \hbox{ for any }k>-1  \end{array}
	\end{equation}

\end{lemma}
 \begin{remark}
	See here  that  the entropic inequalities \eqref{evolentropic+} and \eqref{evolentropic-}   are local and do not proceed up to the boundary like in \cite{Ca}.  Indeed, in contrast of Carillo's approach, we are intending here to handle the boundary condition separately by working with the test functions $\xi_h$ combined with the outward pointing condition \eqref{HypVstg}.   In particular, as we'll see this would reduce the technicality of de-doubling variables  process (see the proof Proposition  \ref{PKato}).   
\end{remark}

\begin{proof}[Proof of Lemma \ref{pentropic}] Let us  consider
	$$\beta_\epsilon(r)=  \tilde \heps (r-k),\quad \hbox{ for any }r\in \RR, $$
	where
	$$\tilde \heps (r)=\left\{ \begin{array}{ll}   \frac{1}{2\eps} {r^+}^2 \quad & \hbox{ for } r\leq \eps \\
		r-\frac{\eps}{2} & \hbox{ elsewhere }.
	\end{array}    \right.  $$
	In this case,   $\beta_\epsilon'(1)=\heps^+ (1)= 1$, for any $0<\epsilon\leq \epsilon_0\hbox{ (small enough)},$ and  we have
	$$  \beta_\eps'(u) =  \hepsp(u-k)\to \spo(u-k),\quad \hbox{ as }\eps\to 0.$$
  For any  $   k< 1,$ since 
	$$\spo (u-k) \:  \chi_{[p=0]}  + \spo(p)  =   \spo(u-k),$$
	letting $\eps\to 0,$ in \eqref{renormal+} where we replace $\beta$ by $\tilde \heps,$ we get
	\begin{equation}
		\begin{array}{c}
			\partial_t   (u-k)^+   - \Delta p^+  + 	V\cdot \nabla (u-k)^+  +u\:  \nabla \cdot V \: \spo (u-k)       \\  \\
			\leq   f \:  \spo(u-k)
			\quad \hbox{ in }  \D'(Q),
	\end{array}  	\end{equation}
	which implies \eqref{evolentropic+}. \\
	 For the second part of the Lemma, we use again the fact that see that   $(\tilde u:=-u,\tilde p:= -p)$ is   a weak solution of  \eqref{cmef} with  $f$ replaced by $\tilde f:=-f.$  So, using
	\eqref{evolentropic+} for $(\tilde u,\tilde p)$ with $\tilde f,$ for any $k< 1,$ we have
	\begin{equation}
		\begin{array}{c}  	-\Delta \tilde p ^+ 	+	 \nabla \cdot(  (\tilde u-k)^+  V)  + k   \:  \nabla \cdot V \:    \spo(\tilde u-k)       \\  \\     \leq    \tilde f \:    \spo(\tilde u-k)  \quad \hbox{ in }  \D'(Q).
		\end{array}
	\end{equation}
	This implies that, for any $-1< s ,$ we have
	\begin{equation}
		\begin{array}{c}  	\partial_t   (s-u)^-	-\Delta   p^-	+	 \nabla \cdot(  (s-u)^+  V)  -s  \:  \nabla \cdot V \:    \spo(s-u)        \\  \\  \leq  -f\:    \spo(s-u)    \quad \hbox{ in }  \D'(Q).  \end{array}
	\end{equation}
Thus the result of the lemma.
\end{proof}

\bigskip
Now, we are    able   to prove  some kind of ''local'' Kato's inequality which generates  our $L^1-$approach for the uniqueness.

 \bigskip
 \begin{proposition}[Kato's inequality]\label{PKato}
 Under the assumptions of Section \ref{sectionprem},  if   $(u_1,p_1)$  and  $(u_2,p_2)$   are two  couples  of  $L^\infty(Q) \times  L^2	\left(0,T;H^1_D(\Omega)\right) $   satisfying  \eqref{evolentropic+}  and  \eqref{evolentropic-} corresponding to $f_1\in L^1(Q)$ and  $f_2\in L^1(Q)$ respectively, then   there exists $\kappa\in L^\infty(Q),$  such that  $\kappa\in \signp(u_1-u_2)$ a.e. in $Q$ and 
 	\begin{equation}\label{ineqkato}
 		\begin{array}{c}
 			\partial_ t	( u_1-u_2 )^+
 			- \Delta      (  p_1^+  +  p_2 ^-)  + \nabla \cdot \left(
 			( u_1-u_2 )^+  \: V \right)      \leq  \kappa  ( f_1-f_2 )  \,
 			\hbox{ in }\D'(Q).
 		\end{array}
 	\end{equation}
 \end{proposition}
 \begin{proof}
 	The proof of this lemma is based on  doubling and de-doubling variable techniques.  Let us give here briefly the arguments.  To double the variables, we fix $\tau>0$, and since $u_1(s,y)-\tau<1,$ we use the fact that $(u_1,p_1)$ satisfies  \eqref{evolentropic+}  with $k=u_1(s,y)-\tau, $ we have
 	\begin{equation}
 		\begin{array}{c}  \frac{d}{dt}   \int     (u_1(t,x)-u_2(s,y) +\tau )^+   \:
 			\zeta  + \int   (\nabla_x   p_1^+ (t,x) - ( u_1(t,x)-u_2(s,y) +\tau )^+ \:   V(x)   \cdot   \nabla_x\zeta     \\      + \int_\Omega  \nabla_x \cdot V \:  u_2(s,y)  \zeta   \spo   (u_1(t,x)-u_2(s,y) +\tau  )  \leq \int f_1(t,x) \spo   (u_1(t,x)-u_2(s,y) +\tau )  \: \zeta, \end{array}
 	\end{equation}
 for any $0\leq \eta\in \D(\Omega\times \Omega).$ 	See that   $ \int    \nabla_y  p_2 ^- (s,y)   \cdot   \nabla_x\zeta  \: dx =0,$ so that
 	\begin{equation}
 		\begin{array}{c}   \frac{d}{dt}   \int     (u_1(t,x)-u_2(s,y) +\tau )^+    \:
 			\zeta  + \int  ( \nabla_x   p_1^+  (t,x) +\nabla_y   p_2 ^-  (t,x) )   \cdot   \nabla_x\zeta -  \int    ( u_1(t,x)-u_2(s,y) +\tau )^+  \:   V(x)   \cdot   \nabla_x\zeta    \\   +  \int  \nabla_x \cdot V \:  u_2(s,y) \:  \zeta\:    \spo   (u_1(t,x)-u_2(s,y) +\tau ) \leq    \int  f_1(t,x) \spo    (u_1(t,x)-u_2(s,y) +\tau)   \: \zeta     . \end{array}
 	\end{equation}
 	Denoting  by
 	$$u(t,s,x,y)=  u_1(t,x) -u_2(s,y)+\tau  ,\quad \hbox{ and }  \quad    p(t,s,x,y)=   p_1^+  (t,x) +  p_2 ^-  (s,y) ,$$
 	and  integrating with respect to $y,$ we obtain
 	\begin{equation}
 		\begin{array}{c}  \frac{d}{dt}   \int    u(t,s,x,y)^+   \:
 			\zeta  +   \int\!\! \int  (\nabla_x+\nabla_y)  p(t,s,x,y)   \cdot   \nabla_x\zeta -   \int\!\! \int  u(t,s,x,y) ^+  \:   V(x)   \cdot   \nabla_x\zeta    \\    +  \int\!\! \int    \nabla_x \cdot V \:  u_2(s,y) \:  \zeta\:    \spo   u(t,s,x,y)    \leq     \int\!\! \int   f_1(t,x) \spo   u(t,s,x,y)    \: \zeta    . \end{array}
 	\end{equation}
 	On the other hand, since $u_1(t,x)+\tau >-1,$ using the fact that $(u_2,p_2 )$ satisfies  \eqref{evolentropic-} with $k=u_1(t,x)+\tau,$ we have
 	\begin{equation}
 		\begin{array}{c}
 			\frac{d}{ds}   \int    u(t,s,x,y)^+   \:
 			\zeta      + 	  \int   (\nabla_y   p_2 ^- (s,y) - u(t,s,x,y) ^+ \:   V(y)   \cdot   \nabla_y\zeta      \\    - \int_\Omega  \nabla_y \cdot V \:  u_1(t,x)  \zeta   \spo  (u(t,s,x,y))   \leq - \int f_2(s,y) \spo ( u(t,s,x,y) ) \: \zeta   .\end{array}
 	\end{equation}
 	Working in the same way,    we get
 	\begin{equation}
 		\begin{array}{c}   
 			\frac{d}{ds}   \int    u(t,s,x,y)^+  \:
 			\zeta      +    \int\!\! \int   (\nabla_x+\nabla_y)  p(t,s,x,y)     \cdot   \nabla_y\zeta -  \int\!\! \int  u(t,s,x,y) ^+  \:   V(y)   \cdot   \nabla_y\zeta   \\     -   \int\!\!  \int   \nabla_y \cdot V(y) \:  u_1(t,x)   \:  \zeta\:    \spo (  u(t,s,x,y) )   \leq   -  \int\!\! \int   f_2(s,y)  \spo   (u(t,s,x,y) )  \: \zeta       . \end{array}
 	\end{equation}
 	Adding both inequalities, we obtain
 	\begin{equation}\label{formdoubling1}
 		\begin{array}{c}   \left(   \frac{d}{dt}  +   \frac{d}{ds}   \right)    \int\!\! \int u(t,s,x,y) ^+   \:
 			\zeta  + \int\!\! \int   (\nabla_x+\nabla_y)   p(t,s,x,y)    \cdot   (\nabla_x +\nabla_y)\zeta  \\ -  \int\!\! \int    u(t,s,x,y) ^+  \: (   V(x)   \cdot    \nabla_x\zeta +V(y)   \cdot  \nabla_y\zeta)    \\  +  \int\!\!  \int  \left(   \nabla_x \cdot V(x) \:  u_2(s,y)    -   \nabla_y \cdot V(y) \:  u_1(t,x)    \right) \:  \zeta\:    \spo   (u(t,s,x,y) )  \\  
 			\leq   \int\!\!  \int  ( f_1(t,x) -  f_2(s,y) )  \spo ( u(t,s,x,y) )   \: \zeta    .\end{array}
 	\end{equation}
 	 	Now, we can de-double the variables $t$ and $s,$ as well as $x$ and $y,$ by taking as usual the sequence of test functions  
 	\begin{equation}    
 		\psi_\eps(t,s) = \psi\left( \frac{t+s}{2}\right)   \rho_\eps \left( \frac{t-s}{2}\right)    \hbox{ and }  
 		\zeta_\lambda  (x,y) = \xi\left( \frac{x+y}{2}\right) \delta_\lambda \left( \frac{x-y}{2}\right),
 	\end{equation}
 	for any $t,s\in   (0,T)$ and $x,y\in \Omega.$  Here   $\xi\in \D(\Omega),$   $\psi\in \D(0,T),$  $\rho_\eps$  is a sequence of usual  mollifiers in $\RR,$ and  $\delta_\lambda$ is the  sequence of    mollifiers in $\RR^N$  given for instance by \eqref{molifiers}. 
See that 
 	$$   \left(   \frac{d}{dt}  +   \frac{d}{ds}   \right)   \psi_\eps(t,s)  =  \rho_\eps\left( \frac{t-s}{2}\right)  \dot   \psi \left( \frac{t+s}{2}\right)   $$ 
 	and 
 	$$(\nabla_x+ \nabla_y)  \zeta_\lambda  (x,y)  =   \delta_\lambda \left( \frac{x-y}{2}\right)  \nabla \xi\left( \frac{x+y}{2}\right)  $$
 	Moreover,   for  any $h\in L^1((0,T)^2\times \Omega^2)$ and   $\Phi\in L^1((0,T)^2\times  \Omega^2)^N,$   it is not difficult to prove that 
 	\begin{itemize}
 		\item  $\lim_{\lambda \to 0 }\lim_{\eps \to 0 } \int_0^T\!\! \int_0^T\!\!   \int_\Omega\!\!  \int_\Omega h(t,s,x,y)\:\zeta_\lambda(x,y)\: \psi_\eps(t,s)   = \int_0^T\!\!  \int_\Omega  h(t,t,x,x)\: \xi(x)\:  \psi(t) .$
 		
 		\item  $\lim_{\lambda \to 0 }\lim_{\eps \to 0 } \int_0^T\!\! \int_0^T\!\!   \int_\Omega\!\!  \int_\Omega h(t,s,x,y)\:\zeta_\lambda(x,y)\: \left(\frac{d}{dt} + \frac{d}{ds}\right)\psi_\eps(t,s)  = \int_0^T\!\!   \int_\Omega  h(t,t,x,x)\: \xi(x)\:  \dot \psi(t) .$ 
 		
 		\item $\lim_{\lambda \to 0 }\lim_{\eps \to 0 }   \int_0^T\!\!  \int_\Omega\!\!  \int_\Omega \Phi(t,s,x,y) \cdot  (\nabla_x + \nabla_y) \zeta_\lambda(x,y) \: \psi_\eps(t,s)  = \int_0^T\!\!\int_\Omega  \Phi(t,t,x,x)\cdot \nabla \xi(x)\: \psi(t)\: dtdx   .$
 	\end{itemize}  
 	So replacing $\zeta$ in \eqref{formdoubling1} by $\zeta_\lambda$,  testing with $\psi_\eps$  and,  letting $\eps\to 0$ and $\lambda \to 0,$ we get  
 	\begin{equation}\label{de-double }
 		\begin{array}{c}
 		-\int_0^T\!\! \int_\Omega \Big\{  (u_1-u_2)^+   \:
 			\xi \: \dot \psi   +  \nabla (p_1^++p_2^-)  \cdot \nabla \xi\:   \psi  
 			-      (u_1-u_2+\tau )^+  \: \left(    V    \cdot   \nabla \xi  \:   +       \nabla \cdot V    \: \xi \: \right)  \psi  \Big\}  \\   \leq   \int_0^T\!\! \int_\Omega  \kappa_\tau  (x) (f_1-f_2) \: \xi   \:   \psi   +
 			\lim_{\lambda \to 0 } \int_0^T\!\!  \int_\Omega \!\!\int_\Omega     u(t,t,x,y)^+    \:    (V(x) -V(y))   \cdot \nabla_y  \zeta  \:   \psi     , \end{array}
 	\end{equation}
 where $\kappa_\tau \in L^\infty(Q)$ is such that $\kappa_\tau \in \signp(u_1-u_2+\tau ) $ a.e. in $Q.$ 
  Thus 
 \begin{equation}\label{dedouble}
 	\begin{array}{c}
 		\frac{d}{dt}    \int (u_1-u_2+\tau )^+   \:
 		\xi   + \int  \nabla (p_1^++p_2^-)  \cdot \nabla \xi
 		-   \int  (u_1-u_2+\tau )^+  \:  \left(   V    \cdot   \nabla \xi  +    \nabla \cdot V     \: \xi \right) \\   \leq   \int \kappa (x) (f_1-f_2) \: \xi   +
 		\lim_{\lambda \to 0 } \int_\Omega \!\!\int_\Omega     u(t,t,x,y)^+    \:    (V(x) -V(y))   \cdot \nabla_y  \zeta . 
 	 \end{array}
 \end{equation} 
  To pass to the limit in the last term corresponding to the vector filed $V,$ we use  moreover the technical result of Lemma \ref{Ltech1} (cf. Appendix)  which  is more or less well known. We put back its  proof in the Appendix.  This implies that  
 	\begin{equation}
 		\begin{array}{c}
 			\frac{d}{dt}    \int ( u_1(t,x) -u_2(t,x)+\tau )^+  \:
 			\xi \: dx   + \int    \nabla    (p_1^+  +   p_2 ^- )\cdot \nabla \xi\: dx  
 		  \\  	-   \int    ( u_1-u_2+\tau)^+  \:    V    \cdot   \nabla \xi  \: dx     \leq   \int \kappa_\tau ( f_1-f_2 )\: \xi\: dx 
 			. \end{array}
 	\end{equation}
 	Letting then $\tau\to 0,$ and using again the fact that  $\kappa_\tau\to \kappa$ weakly to $L^\infty(Q)$, with $\kappa\in \sign^+(u_1-u_2),$ a.e. in $Q,$ the result of the proposition follows. 
  \end{proof}

 \bigskip 
 The aim now  is to  process  with the sequence of test function $\xi_h$ given by \eqref{xih} in Kato's inequality and let $h\to 0,$ to cover  \eqref{evolineqcomp}. 
 
 \begin{lemma}\label{ltechevol1} Under the assumptions of  Section \ref{sectionprem}, if  $(u,p)$ is  a  weak solution of \eqref{cmef}, then for any $0\leq \psi\in D(0,T)$,  we have
 	\begin{equation} \label{signp+}
 		\liminf_{h\to 0} \int_0^T\!\!\int_\Omega \nabla   p^\pm   \cdot \nabla \xi_h \: \psi(t)\: dtdx \geq 0.
 	\end{equation}
 
 \end{lemma}
 \noindent\textbf{Proof : } Using \eqref{eqp+1}, we see that,   for any $0\leq \psi\in D(0,T)$  we have 
  \begin{eqnarray}
 \int_0^T\!\! 	\int_\Omega  \nabla   p^+ \cdot \nabla   \xi_h \: \psi\: dtdx   &=&   -   \int_0^T\!\! 	\int_\Omega  \nabla   p^+  \cdot \nabla  (1- \xi_h)\: \psi\: dtdx   \\  \\ 
	&\geq&   \int_0^T\!\! 	\int_\Omega  ( \nabla \cdot V - f)   (1- \xi_h)     \spo(p)\: \psi\: dtdx    
	 \end{eqnarray} 
Then, letting $h\to 0$ and using the fact that  	$   \xi_h   \to 1 \hbox{    in }L^\infty(\Omega)-\hbox{weak}^*,$  we deduce that    $ 	\liminf_{h\to 0}  \int_0^T\!\! 	\int_\Omega  \nabla   p^+ (t) \cdot \nabla   \xi_h \: \psi\: dtdx   \geq 0.$ 
The proof for $p^-$   follows in a similar way  by using   \eqref{eqp-1}. 
 
  \qed

 \begin{lemma}\label{pintegral} Under the assumptions of  Theorem \ref{compcmef},  
if  $(u_1,p_1)$  and  $(u_2,p_2)$   are two  couples  of  $L^\infty(Q) \times  L^2	\left(0,T;H^1_D(\Omega)\right) $   satisfying  \eqref{evolentropic+}  and  \eqref{evolentropic-} corresponding to $f_1\in L^1(Q)$ and  $f_2\in L^1(Q)$ respectively, then   there exists $\kappa\in L^\infty(Q),$  such that  $\kappa\in \signp(u_1-u_2)$ a.e. in $Q$ and \eqref{evolineqcomp} is fulfilled.
 \end{lemma}
 \begin{proof}
 	See that
 	\begin{eqnarray}
 	  \frac{d}{dt}	\int	( u_1-u_2 )^+ \: dx - \int \kappa (f_1-f_2) \: dx   &=&   \lim_{h\to 0}  \underbrace{\frac{d}{dt} 	\int
  	( u_1-u_2 )^+    \, \xi_h \: dx    -\int   \kappa (f_1-f_2)   \xi_h  \: dx}_{ I(h) }.  
 	\end{eqnarray}
 	Taking  $\xi_h$ as a test function in \eqref{ineqkato},  we have
 	\begin{eqnarray}\label{inqxih}
 		I(h)
 	&	\leq&  -    \int \left(  \nabla  (  p_1^+  +  p_2^-)   -
 		( u_1-u_2 )^+  \: V \right)  \cdot\nabla  \xi_h  \: dx  \\ 
 		&\leq& -    \int   \nabla  (  p_1^+  +  p_2^-)     \cdot\nabla  \xi_h  \: dx -  \int
 		( u_1-u_2 )^+  \: V   \cdot \nu_h(x) \: dx    . 
	 	\end{eqnarray}	
Using Lemma \ref{ltechevol1}, this implies that  
 	\begin{equation} \label{mainuniq} 
 	\begin{array}{ll}
 			\lim_{h\to 0}I(h)  	&\leq     -	\lim_{h\to 0}   \int
 		( u_1-u_2 )^+  \: V   \cdot \nu_h(x) \: dx  \\  \\
 		&\leq 0
 		\end{array}
 	\end{equation}
 	where we use the outgoing  velocity vector field  assumption \eqref{HypVstg}.  Thus \eqref{evolineqcomp}.
 \end{proof}

 \bigskip
 \begin{proof}[Proof of Theorem \ref{compcmef}]  It clear that the first part follows by Lemma \ref{pintegral}. The rest of the theorem  is a straightforward  consequence of  \eqref{evolineqcomp}. 
 	\end{proof}

 \begin{remark}\label{RemboundaryCond2}
See that Lemma \ref{pintegral}  is the lonely step of the proof of Theorem \ref{compcmef} where we use assumption \eqref{HypVstg} (see \eqref{mainuniq}).  Working so   enables to avoid all the technicality related to 
doubling and de-doubling variable by using test functions which does not vanish on the boundary.  
 We do believe that the result of Theorem \ref{compcmef} remains to be true   under the general assumption \eqref{HypV0}.   We think that this works at least for Dirichlet boundary condition following Carrillo's techniques (cf. \cite{Car}) but it could  be  heavy and very technical.  One sees also that, if the solutions have a trace (like for $BV$ solution), maybe one can  weaken this condition by handling \eqref{mainuniq} otherwise.  \end{remark}

 \section{Existence  for the evolution problem}
 \setcounter{equation}{0}

Throughout this section, we assume that  the assumptions of Section \ref{sectionprem} and   \eqref{HypVstg} are fulfilled. Our main result here concerns existence of a solution.

\begin{theorem}\label{existcmef}
	For any $f\in L^2(Q)$ and $u_0\in L^\infty(\Omega)$ be such that $\vert u_0\vert \leq 1$ a.e. in $\Omega,$ the problem \eqref{cmef} has a unique   solution  $u\in \C([0,T),L^1(\Omega))$ and $u(0)=u_0$ in the sense of Definition \ref{defws}. Moreover, the solution satisfies all the properties of Theorem \ref{compcmef}. 
\end{theorem}

 Thanks to   Theorem \ref{existcmef} and  Theorem \ref{compcmef}, we have the following practical  particular results.  
 \begin{corollary}
 	Let $f\in L^2(Q),$   $u_0\in L^\infty(\Omega)$ be such that $\vert u_0\vert \leq 1$ a.e. in $\Omega,$ and $u$ be   the unique weak solution of the problem \eqref{cmef}.  
 	
 	\begin{enumerate}
 		\item  If  $u_0\geq 0$ and $f\geq 0,$ then $(u,p)$ is the unique  weak solution of the one phase problem
 		\begin{equation} 			\left\{  \begin{array}{ll}\left.
 				\begin{array}{l}
 					\displaystyle \frac{\partial u }{\partial t}  -\Delta p +\nabla \cdot (u  \: V)= f \\
 					\displaystyle  0\leq u\leq 1,\:  p\geq 0,\:  p(u-1)=0 \end{array}\right\}
 				\quad  & \hbox{ in } Q \\  
 				\displaystyle p= 0  & \hbox{ on }\Sigma_D \\  
 				\displaystyle (\nabla p- u  \: V)\cdot \nu = 0  & \hbox{ on }\Sigma_N  \\ 
 				\displaystyle  u (0)=u _0 &\hbox{ in }\Omega,\end{array} \right. 
 		\end{equation}
 		in the sense of Definition \ref{defws}.

 		\item For each $n=1,2,....,$ let  $f_n\in L^1(Q),$   $u_{0n}\in L^\infty(\Omega)$ be such that $\vert u_{0n}\vert \leq 1$ a.e. in $\Omega,$ and $u_n$   be  the unique weak solution of the corresponding problem \eqref{cmef}. If $u_{0n}\to u_0$ in $L^1(\Omega)$ and $f_n\to f$ in $L^1(Q),$ then $	u_n \to u,$  in $ \C([0,T),L^1(\Omega))$,    $	p_n \to p, $ in $ L^2(0,T;H^1_D(\Omega))-\hbox{weak},$ and $(u,p)$ is the solution corresponding to $u_0$ and $f.$

 	\end{enumerate} 
 \end{corollary}

To study  the existence, we  process  in the framework of nonlinear semigroup theory in $L^1(\Omega).$ In connection with the Euler implicit discretization scheme of the evolution problem \eqref{cmef}, we consider  the  following  stationary problem   :
 \begin{equation}	\label{st}
 	\left\{  \begin{array}{ll}\left.
 		\begin{array}{l}
 			\displaystyle u - \Delta p +  \nabla \cdot (u  \: V)=f \\
 			\displaystyle u \in \hbox{Sign}(p)\end{array}\right\}
 		\quad  & \hbox{ in }  \Omega\\  \\
 		\displaystyle p= 0  & \hbox{ on }  \Gamma_D\\  \\
 		\displaystyle (\nabla p- u  \: V)\cdot \eta = 0  & \hbox{ on }  \Gamma_N ,\end{array} \right.
 \end{equation}
 where $\displaystyle f \in L^2(\Omega)$ and $\lambda >0$ are given.  A couple $(u ,p) \in  L^\infty(\Omega) \times H^1_D(\Omega)$  satisfying   $\displaystyle    u \in \sign(p)$ a.e. in $\Omega,$ 
 is said to be a weak solution  of \eqref{st}  if
 \begin{equation}\label{stws}
 	\displaystyle  \int_\Omega u \:\xi+  \int_\Omega  \nabla p \cdot  \nabla\xi  -
 	\int_\Omega  u \:  V\cdot \nabla  \xi    =     \int_\Omega f\: \xi,\quad \forall\: \xi\in H^1_D(\Omega).
 \end{equation}
  As for the evolution problem, we'll say simply that $u$ is a solution of \eqref{st} if there exists $p$ such that the couple $(u,p)$ is a weak solution of   \eqref{st}.

 \bigskip
 As a consequence of  Theorem \ref{compcmef}, we can deduce the following result.

 \begin{corollary}\label{ccontractionst}
 	If  	$u _1$ and $u _2$ are   two   solutions of  \eqref{st} associated with $f_1,\: f_2 \in L^1(\Omega),$  respectively, then
 	\begin{equation}
 		\label{ineqcomp}
 		\int( u _1-u _2)^+ \leq \int ( f_1-f_2)^+ .
 	\end{equation}
 	In particular, we have
 	$$ \Vert u_1-u_2\Vert_{1} \leq \Vert f_1-f_2\Vert _{1}$$
 	and, if  $f_1\leq f_2,$  a.e. in  $\Omega,$   then
 	$$u_1\leq u_2,\quad \hbox{ a.e. in  }\Omega.$$
 \end{corollary}
 \begin{proof}
 	This is a simple consequence of the fact  that if $(u,p)$ (which is independent of $t$) is a solution of \eqref{st}, then it  can be assimilated to a time-independent solution of  the evolution problem \eqref{cmef} with $f$ replaced by $f-u$ (which is also independent of $t$).  \end{proof}

 \bigskip For the existence,    we consider the regularized problem
 \begin{equation}	\label{psteps}
 	\left\{  \begin{array}{ll}\left.
 		\begin{array}{l}
 			\displaystyle u - \Delta p +  \nabla \cdot (u  \: V)=f \\
 			\displaystyle u =\heps  (p) \end{array}\right\}
 		\quad  & \hbox{ in }  \Omega\\  \\
 		\displaystyle p= 0  & \hbox{ on }  \Gamma_D\\  \\
 		\displaystyle (\nabla p- u  \: V)\cdot \eta = 0  & \hbox{ on }  \Gamma_N,\end{array} \right.
 \end{equation}
for an arbitrary $\eps >0.$   See that for any $\eps>0,$   $ \vert \heps\vert \leq 1$, $\heps$ is  Lipschitz continuous    and satisfies
 $$(I+\heps)^{-1} (r)\to  (I+\sign)^{-1} (r), \hbox{ as }\eps\to 0,  \hbox{ for any }r\in  \RR. $$
 That is $\heps $ converges to $\sign$ in the sense  of resolvent, which is equivalent to the convergence in the sense of graph (cf. \cite{Br}). 

 \bigskip
 \begin{proposition}\label{extsst} 
 	For any $f\in L^2(\Omega)$ and $\eps>0,$ the problem 	\eqref{psteps} has a   weak solution $(u_\eps,p_\eps) $ in the sense that  $p_\eps\in H^1_D(\Omega),$  $  u_\eps=\heps(p_\eps) $   a.e. in $\Omega$     and
 	\begin{equation}
 		\label{weakeps}
 		\int u_\eps \: \xi \: dx + \int \nabla p_\eps \cdot \nabla \xi \: dx  -\int u_\eps \: V\cdot \nabla \xi \: dx =\int f\: \xi\: dx,  \quad \hbox{ for any }\xi\in H^1_D(\Omega).   \end{equation}
  Moreover,    as $\eps\to 0,$ we have 
 	\begin{equation}\label{convuepss}
 		\heps (p_\eps) \to u \quad \hbox{ in }  L^\infty(\Omega)-\hbox{weak}^*, \end{equation}
 	 	\begin{equation}\label{convpepss} p_\eps \to p\quad \hbox{ in }  H^1_D(\Omega)-\hbox{weak}  \end{equation}
 	and  $(u,p)$ is the weak  solution of  	\eqref{st}. 
\end{proposition} 
\begin{proof}   
 The  existence of a solution for \eqref{psteps}   is standard. For completeness we give the arguments. Let us denote by $H^{1}_D(\Omega)^* $ the topological dual space of $H^{1}_D(\Omega)$ and $ \langle .,. \rangle$ the associate  dual bracket.    See that the operator $A_\eps\: :\:  H^{1}_D(\Omega)\to  H^{1}_D(\Omega)^*,$ given by
 	$$ \langle A_\eps p,\xi\rangle =  \int\heps (p)\: \xi\: dx + \int \nabla p\cdot \nabla \xi \: dx -\int\heps (p) \: V\cdot \nabla \xi \: dx, \quad \forall\: \xi\in H^1_D(\Omega), $$
 	is  bounded and weakly continuous.  Moreover,  $A_\eps$ is coercive. Indeed, for any $u\in H^{1}_D(\Omega),$ we have
 	\begin{eqnarray}
 		\langle A_\eps p,p\rangle &=& \int\heps (p)\: p \: dx + \int \vert \nabla p\vert^2  \: dx -\int\heps (p) \: V\cdot \nabla p  \: dx \\  \\
 		&\geq&    \int \vert \nabla p\vert^2  \: dx -\int \vert V\cdot \nabla p \vert  \: dx \\
 		&\geq&  \frac{1}{2}  \int \vert \nabla p\vert^2  \: dx - \frac{1}{2}\int  \vert  V\vert^2   \: dx,
 	\end{eqnarray}
 	where we use Young inequality.   So, for any $f\in  H^{1}_D(\Omega)^* $ the problem $A_\eps p=f$ has a solution $p_\eps \in H^1_D(\Omega).$  To let $\eps\to 0,$ we see    that 
 	 		\begin{equation} 		\label{hestimate0}
 		\int \vert \nabla p_\eps \vert^2\: dx  \leq  C (N,\Omega) \left(  	\int \vert V  \vert^2\: dx   + 	\int \vert f\vert^2\: dx \right ).
 	\end{equation}
Indeed, taking $ p_\eps $ as a test function we have
 	\begin{equation} \label{testphi}
 		\int u_\eps\:  p_\eps\: dx + \int \vert \nabla  p_\eps\vert^2\: dx =  \int u_\eps\: V\cdot \nabla  p_\eps  \: dx + \int f\:  p_\eps\: dx.
 	\end{equation}
 	See  that using    Young inequality we have
 	$$	\int u_\eps\: V\cdot \nabla  p_\eps \: dx \leq   \frac{3}{4}  \int  \vert V\vert ^2  + \frac{1}{3} \int \vert \nabla p_\eps\vert^2\: dx  $$
 	and, by combining Pincar\'e  inequality with  Young inequality we have
 	$$\int f\:  p_\eps\: dx \leq  C(N,\Omega)  \int  \vert f \vert ^2  + \frac{1}{3} \int \vert  \nabla p_\eps\vert^2\: dx   .$$
 	Using the fact that  $u_\eps\: p_\eps\geq 0,$  we deduce    \eqref{hestimate0}.  Now, it is clear that 
 the  sequences  $p_\eps$ and $u_\eps=\heps (p_\eps)$  are bounded respectively in $H^1_D(\Omega) $ and in  $L^\infty(\Omega)$.  So, there exists a subsequence that we denote again by $p_\eps$ such that \eqref{convuepss} and \eqref{convpeps} are fulfilled.  	In particular, using monotonicity argument (see for instance \cite{Br})  this implies that $u\in \sign(p),$ a.e. and letting $\eps\to 0$ in \eqref{weakeps}, we obtain that $(u,p)$ is a weak solution of \eqref{st}. 
 \end{proof}

 To prove the existence, of a weak solution to \eqref{cmef},  we fix $f\in L^2(Q),$ and for an arbitrary $0<\eps\leq \eps_0,$ and $n\in \NN^*$ be such that  $n\eps=T,$  we consider the sequence of $(u_i,p_i)$ given by the the $\eps-$Euler implicit scheme associated with \eqref{cmef} :
 	\begin{equation}	\label{sti}
 	\left\{  \begin{array}{ll}\left.
 		\begin{array}{l}
 			\displaystyle u_{i+1} - \eps \: \Delta p_{i+1}  +  \eps \: \nabla \cdot (u_{i+1}   \: V)=u_{i} +\eps  \: f_i  \\
 			\displaystyle u_{i+1}  \in \hbox{Sign}(p_{i+1} )\end{array}\right\}
 		\quad  & \hbox{ in }  \Omega\\  \\
 		\displaystyle p_{i+1} = 0  & \hbox{ on }  \Gamma_D\\  \\
 		\displaystyle (\nabla p_{i+1} - u_{i+1}   \: V)\cdot \eta = 0  & \hbox{ on }  \Gamma_N ,\end{array} \right. \quad i=0,1, ... n-1,
 \end{equation}
 where,   for each $i=0,...n-1,$ $f_i$ is given by
 $$f_i = \frac{1}{\eps } \int_{i\eps}^{(i+1)\eps }  f(s)\: ds ,\quad \hbox{ a.e. in }\Omega.  $$
 Now, for a given $\eps-$time discretization $0=t_0<t_1<t_1<...<t_n= T,$ satisfying
 $t_{i+1}-t_i = \eps,  $  we define the $\eps-$approximate solution by
 \begin{equation}
	u_\eps:=  \sum_{i=0} ^{n-1 }  u_i\chi_{[t_i,t_{i+1})},\quad \hbox{ and    }  \quad 	p_\eps:=   \sum_{i=1} ^{n -1}  p_i\chi_{[t_i,t_{i+1})}.
\end{equation}

Thanks to Proposition \ref{extsst} and  the general theory of evolution problem governed by accretive operator (see for instance \cite{Benilan,Barbu}), we define the operator $\A$ in $L^1(\Omega),$ by $\mu\in \A(z)$ if and only if $\mu,\: z\in L^1(\Omega)$ and $z$ is a solution of the problem 

  \begin{equation}	\label{eqop}
 	\left\{  \begin{array}{ll}\left.
 		\begin{array}{l}
 			\displaystyle  - \Delta p +  \nabla \cdot (z  \: V)=\mu\\
 			\displaystyle z \in \hbox{Sign}(p)\end{array}\right\}
 		\quad  & \hbox{ in }  \Omega\\  \\
 		\displaystyle p= 0  & \hbox{ on }  \Gamma_D\\  \\
 		\displaystyle (\nabla p- u  \: V)\cdot \eta = 0  & \hbox{ on }  \Gamma_N ,\end{array} \right.
 \end{equation}
  in the sense that $z\in L^\infty(\Omega),$ there exists  $p \in    H^1_D(\Omega)$  satisfying   $\displaystyle    z \in \sign(p)$ a.e. in $\Omega$ 
and 
 \begin{equation} 
 	\displaystyle  \int_\Omega  \nabla p \cdot  \nabla\xi  -
 	\int_\Omega  z \:  V\cdot \nabla  \xi    =     \int_\Omega \mu \: \xi, \quad \forall \:  \xi\in H^1_D(\Omega).
 \end{equation}  
 
 As a consequence of Corollary \ref{ccontractionst}, we know that  the  operator $\A$ is accretive in $L^1(\Omega)$. Moreover, it is not difficult to see that $\overline{\D(A)}=\{ u\in L^\infty(\Omega)\:  :\:  \vert u\vert \leq 1\hbox{ a.e. in }\Omega \}.$ So,  thanks to the general   theory of nonlinear semigroup governed by accretive  operator (see for instance \cite{Barbu}),  we know that, as $\eps\to 0,$  
 		\begin{equation}\label{convueps}
 			u_\eps \to u,\quad \hbox{ in } \C([0,T),L^1(\Omega),
 		\end{equation}
 		and  $u$ is the so called ''mild solution'' of the evolution problem
 		\begin{equation}\label{cauchy}
 			\left\{\begin{array}{ll}
 				u_t + \A u\ni f \quad & \hbox{ in }(0,T)\\  \\
 				u(0)=u_0.
 			\end{array}  \right.
 		\end{equation}

To accomplish the proof of existence  for the problem  \eqref{cmef}, we prove that the mild solution $u$   is in fact the solution of \eqref{cmef}. To this aim, we use   the limit of the sequence $p_\eps$ given by the $\eps-$approximate solution. 
\begin{lemma}\label{lmildweak}
	As $\eps\to 0,$
	\begin{equation}\label{convpeps}
		p_\eps \to p,\quad \hbox{ in } L^2(0,T; H^1_D(\Omega)) 
	\end{equation}
and  $(u,p)$ is  a weak solution of \eqref{cmef}. 
\end{lemma}
\begin{proof}
	Thanks to Proposition \ref{hestimatei}, the  sequence $(u_i,p_i)_{i=1,...n}$ given by  \eqref{sti} is well defined in $L^\infty(\Omega)\times H^1_D(\Omega)$   and satisfies $\displaystyle    u_i \in \sign(p_i)$  and
	\begin{equation}\label{stwsi}
		\displaystyle  \int_\Omega u_{i+1} \:\xi+   \eps\: \int_\Omega  \nabla p_{i+1}  \cdot  \nabla\xi  -  \eps\:
		\int_\Omega  u_{i+1}  \:  V\cdot \nabla  \xi    =   \eps  \int_\Omega f_i\: \xi,\quad \hbox{ for any }\xi\in H^1_D(\Omega) .
	\end{equation}
Taking $p_{i+1}$ as a test function in \eqref{stwsi}, working as for the proof  of \eqref{hestimate0} and using the fact that $(u_{i+1}-u_i)p_{i+1}\geq 0,$ we get 
\begin{equation} 		\label{hestimatei}
	\int \vert \nabla p_i\vert^2\: dx  \leq    C (N,\Omega)   \left(  	\int \vert V  \vert^2\: dx   + 	\int \vert f_i\vert^2\: dx \right ).
\end{equation}
   Thus 
\begin{equation}
	\int \vert \nabla p_\eps \vert^2\: dx  \leq    C (N,\Omega)   \left(  	\int \vert V  \vert^2\: dx   + 	\int \vert f_\eps\vert^2\: dx \right ), 
\end{equation}
where $f_\eps =\sum_{i=0}^{n-1} f_{i}\chi_{[t_i,t_{i+1})},\quad \hbox{ in }Q.$ 
This implies that $p_\eps$ is bounded in $L^\infty(0,T;H^1_D(\Omega)),$   and    that there exists
$p\in L^\infty(0,T;H^1_D(\Omega)) $ such that, taking a subsequence if necessary, 
\begin{equation}
	p_\eps \to p,\quad \hbox{ in } L^2(0,T;H^1_D(\Omega))-\hbox{weak}.
	\end{equation}
Combining this with  \eqref{convueps}, we deduce moreover that $u\in \sign(p),$ a.e. in $Q.$  Now, as usually used with   nonlinear semigroup theory for evolution problem, we consider   
\begin{equation}
  \tilde u_\eps =  \sum_{i=0} ^{n-1}  \frac{(t-t_i)u_{i+1} - (t-t_{i+1}) u_i}{\eps } \chi_{[t_i,t_{i+1})},
\end{equation}
which converges to $u$ as well in $\C([0,T);L^1(\Omega)).$ 
We see that  for any test function  $\displaystyle
\xi\in H^1_D(\Omega),$ we have
\begin{equation}
	\displaystyle \frac{d}{dt}\int_\Omega   \tilde u_\eps  \:\xi+\int_\Omega ( \nabla p_\eps -  u_\eps \:V) \cdot  \nabla\xi   =     \int_\Omega f_\eps\: \xi  , \quad \hbox{ in }{\D}'([0,T)).
\end{equation}
So, letting $\eps\to 0,$ and using the convergence of $(\tilde u_\eps,u_\eps,p_\eps,f_\eps)$ to $(u,u,p,f),$ we deduce that $(u,p)$ is a weak solution of \eqref{cmef}.

	\end{proof}

\begin{corollary}\label{cauchyweak}
	For any $f\in L^2(Q)$ and $u_0\in L^\infty(\Omega)$   such that $\vert u_0\vert \leq 1$ a.e. in $\Omega,$ the mild solution of \eqref{cauchy} is the unique solution of   \eqref{cmef}. 
\end{corollary}

 \section{Extension, comments and remarks}
 
 	\begin{enumerate} 
%
 		\item 
 		Thanks to the proof of Proposition \ref{prenormal},  one sees that we can  work with $\xi\in H^1_D(\Omega),$ so that  any weak solution $(u,p)$ satisfies 
 		\begin{equation} \label{renormal+b}	
 			\begin{array}{c}
 				\frac{d}{dt}\int \beta (u)\: \xi\: dx   + \int \nabla  p^+ \cdot \nabla  \xi\: dx     +	\langle V\cdot \nabla \beta(u) )  ,  \xi\rangle + \int  (u\:  \nabla \cdot V   ( 	\beta'(u) 	\: \chi_{[p=0]}   +    \spo(p)  )  \:   \xi\: dx      \\  \\
 				\leq  \int  f \:  (\beta'(u) \:  \chi_{[p=0]}  + \spo(p) ) \:   \xi\: dx 
 				\quad \hbox{ in }  \D'(0,T),
 			\end{array}  	
 		\end{equation}
 		and
 		\begin{equation}
 			\label{renormal-b} 
 			\begin{array}{c}
 				\frac{d}{dt}\int \beta (u)\: \xi\: dx  + \int \nabla  p^- \cdot \nabla  \xi\: dx     +	\langle V\cdot \nabla \beta(u) )  ,  \xi\rangle + \int  (u\:  \nabla \cdot V   ( 		\beta'(u)\:
 				\:  \chi_{[p=0]}  -    \smo(p)  )  \:   \xi\: dx      \\  \\
 				\leq  \int  f \:  ( \beta'(u)\:   \chi_{[p=0]}  -  \smo(p)  ) \:   \xi\: dx 
 				\quad \hbox{ in }  \D'(0,T),
 			\end{array}  
 		\end{equation}
 		for any $0\leq \xi\in H^1_D(\Omega)$ and $\beta\in \mathcal C^1(\RR)$ such that $0\leq \beta'\leq 1.$  In particular, for  the case of Neumann boundary condition ; i.e. $\Gamma_D=\emptyset,$ these formulations imply that entropic inequalities  \eqref{evolentropic+} and  \eqref{evolentropic-}  hold to be true up to the boundary $\partial \Omega$. That is \eqref{evolentropic+} and  \eqref{evolentropic-}  are fulfilled in $\D'((0,T)\times \overline \Omega)).$   However, to de-double the variable by taking into account the boundary becomes unhandy since one needs to handle the trace on the boundary  in a strong sense which is not guaranteed in general  (one can see for instance \cite{AnBo} where this kind of approach is developed for elliptic-parabolic equation with homogeneous boundary condition).   By the way, let us mention here that this case  		 is fruitful   for  gradient flow approach in euclidean Wasserstein space. 
 		An    elegant  and promising proof of uniqueness  is given in  \cite{DM} for the  one phase problem under a monotonicity assumption on $V.$  The strength of this approach is the absolute continuity of   weak solution in  the set of probability equipped with Wasserstein distance.

	\item[] 
	\item  For the case where $\Omega=\RR^N,$ one sees that working as in Section 2, it is possible to prove local Kato's inequality of the type  \eqref{ineqkato} and achieve   $L^1-$comparison principle for weak solution like in Theorem \eqref{compcmef}. However, the existence of weak solutions in  $L^1(\RR^N)$ is not clear yet for us.

	 \item[] 
	\item   Beyond their  significant role  for the uniqueness, in our opinion  formulations \eqref{renormal+} and  \eqref{renormal-}  remain  to be   interesting also for qualitative descriptions of the weak solution. As usual for renormalized formulation, they  enable  to localize the description of the solution with respect to its    values.   For similar description  with respect   to the values of $p,$ one can see the following remarks which aims to widen the  formulations   \eqref{eqp+1} and \eqref{eqp-1}.

 \item[] 
 \item The renormalization process we use in this paper concerns essentially  $u.$  The formulation \eqref{eqp+1} and \eqref{eqp-1}  provide  a particular  renormalization  for $p$ in the region $[p\neq 0].$  Notice here, that it is possible to prove mended renormalized  formulations for $p$.   As a matter of fact, under the assumptions of Lemma \ref{leqp+}, we can prove that   for any $H\in \C^1(\RR)$ such that  $H'\geq 0$ and $ H(0)=0$,  we have
 \begin{equation}\label{renorp}
 	-\nabla \cdot(H(p)\nabla p)+  \vert  H(p) \vert     \:  \nabla \cdot V   =    -H'(p)  \: \vert \nabla p\vert^2  +f\: H(p)  \quad \hbox{ in }\D'(Q).
 \end{equation}
 If moreover, $H\geq 0,$ then  
 \begin{equation}\label{renorpb}
 	-\nabla \cdot(H(p)\nabla p)+  \vert  H(p) \vert     \:  \nabla \cdot V   \leq    -H'(p)  \: \vert \nabla p\vert^2  +f\: H(p)  \quad \hbox{ in }\D'((0,T)\times \overline \Omega). 
 \end{equation}
 These  formulations improve the description of $p$ with respect to the values of $p,$ and yield renormalized formulations for $p.$  As we see in Section 2, the case where $H$ approximates   $\spo$ is crucial for the proof of uniqueness. The general case  may serve out    other qualitative description of the solution.  The  proof of \eqref{renorp} and  \eqref{renorpb}   follows more or less similar arguments  as for the proof  Lemma \ref{leqp+}.

 \item[]   
\item   Coming back to the approximation \eqref{psteps}, thanks to Proposition \ref{extsst}, we know that weak solutions of the stationary problem of  \eqref{psteps}    depend weak-continuously on the non-linearity $\heps.$  For the evolution problem,   working as in the proof of Proposition \ref{extsst}, one can prove that  the problem 	 
				\begin{equation}
					\label{cmefeps}
					\left\{  \begin{array}{ll}\left.
						\begin{array}{l}
							\displaystyle \frac{\partial u }{\partial t}  -\Delta p +\nabla \cdot (u  \: V)= f \\
							\displaystyle u=\heps(p)\end{array}\right\}
						\quad  & \hbox{ in } Q \\  
						\displaystyle p= 0  & \hbox{ on }\Sigma_D \\  
						\displaystyle (\nabla p-u  \: V)\cdot \nu = 0  & \hbox{ on }\Sigma_N \\   
						\displaystyle  u (0)=u _0 &\hbox{ in }\Omega,\end{array} \right.
				\end{equation}
				has a unique weak solution $(u_\eps,p_\eps)\in \C([0,T),L^1(\Omega))\times L^2(0,T;H^1_D(\Omega)),$ with $u_\eps = \heps(p_\eps)$ a.e. in $Q,$ and as $\eps \to 0,$ we have $	u_\eps \to u,$  in $ L^\infty(Q)-\hbox{weak}^*$  and $	p_\eps \to p, $ in $ L^2(0,T;H^1_D(\Omega)). $   We do believe that the convergence of $u_\eps$   remains to be true strongly in $L^1(\Omega),$ for both problems stationary and  evolution one. We did  not get into these further questions here. This need more   fine estimates  on $u$ (like $BV-$estimates).   We'll touch them in forthcoming works. The case where  one replace $\heps$ by an arbitrary non-linearity which converges to $\sign,$ in the sense of resolvent, like for instance the porous medium non-linearity $\varphi(r)= \vert r\vert^{m-1} r,$ for any $r\in \RR,$ with  $m\to \infty,$ is likely concerned.


 	\item[] 
 \item It is clear that our approach is based on a renormalization process  à la DiPerna-Lions. This corroborates that this concept is still  fruitful for the uniqueness of weak solution (solution in the sense of distribution) for nonlinear versions of Fokker-Planck equation of the type \eqref{pdetype}. As we quote in the previous remarks, operating with  the absolute continuity of weak solution in  the set of probability equipped with Wasserstein distance   seems to be  promising (cf. \cite{DM}) even if it is actually hampered by conceptual  technical assumptions like the sign of the solution, conservation of the mass and   monotone   transport field (connected to some kind of convexity conditions). 
 We do not overlook the $L^1-$kinetic approach, which was developed in \cite{PerthameChen} for Cauchy problems of general degenerate parabolic-hyperbolic equations with non-isotropic nonlinearity. The approach could be applied  to more general situations and, as  far as we know, this approach  has not yet been explored for  problems of the type \eqref{pdetype} and \eqref{pdetype+}.     At last let us quote here the very recent work of \cite{Naomi}  (that we  learned just when we finish the preparation of this paper through  Benoit Perthame) where the authors deal among other questions with the uniqueness of weak solutions for the one phase problem \eqref{pdetype+} in $\RR^N,$  by means of Hilbert's duality method. This method turns up to be very restrictive since they need the vector filed $V=\nabla \Phi$ to be smooth enough (at least such that $\nabla \Delta \Phi$ is a $L^{{12}/{5}}_{loc}-$Lebesgue function).

  \item Under the same conditions on $V,$ existence and uniqueness of weak solution remains to be true if the source 
  term is replaced by a reaction terms of the form $g(t,x,u)$ satisfying reasonable assumptions (including continuous Lipschitz dependence with respect to $u$). This can be extended also to the case where the reaction is in the form $g(.,u)G(p).$ These generalizations are treated separately in the paper  \cite{IgHSR}.

 \end{enumerate}

 \section*{Acknowledgments}  
We would like to thank Benoit Perthame for  informing us about the papers \cite{Figali}, \cite{LebrisLions} and \cite{Naomi}.  
 
 	\section{Appendix}
\setcounter{equation}{0}

The aim of this section is to complete the proof of Proposition \ref{PKato} with the proof of   de-doubling variable process concerning the vector field part of  \eqref{dedouble}, which is 
	$$\lim_{\lambda \to 0 } \int_\Omega \!\!\int_\Omega     u(t,t,x,y)^+    \:    (V(x) -V(y))   \cdot \nabla_y  \zeta_\lambda\: dxdy.$$ 
Remember that 
	\begin{equation}    
	\zeta_\lambda  (x,y) = \xi\left( \frac{x+y}{2}\right) \delta_\lambda \left( \frac{x-y}{2}\right),
\end{equation} 
for any $x,y\in \Omega,$  where  $\xi\in \D(\Omega)$  and   $\delta_\lambda$ is a  sequence  of    mollifiers in  $\RR^N$ given by 
 \begin{equation} \label{molifiers}
 \delta_\lambda (x)=\frac{1}{\lambda ^N}  \delta\left (\frac{\vert x\vert ^2}{\lambda^2 }\right),\quad \hbox{ for any }x\in \RR^N ,\end{equation}
  where $0\leq \delta \in \D(0,1)$ and $\frac{1}{2^N}  \int \delta(\vert z\vert^2)\: dz =1  .$     A typical example is given by $ \delta (r)=   \: Ce^{\frac{-1}{1-r ^2 } }\chi_{[\vert r\vert \leq 1]},$ for any $r\in\RR^N.  $

\begin{lemma} \label{Ltech1}
	For  any $h\in L^1((0,T)\times \Omega\times \Omega)$,   $V\in W^{1,1}(\Omega)$ and $\psi\in \D(0,T),$   we have 	   
	$$\lim_{\lambda \to 0 } \int_0^T\!\!  \int_\Omega\!\!  \int_\Omega    h(t,x,y) \:    (V(x) -V(y))   \cdot \nabla_y  \zeta_\lambda(x,y)  \:   \psi(t)\: dtdxdy$$
	\begin{equation} \label{inqtech1}
		=  	    \int_0^T\!\!\int_\Omega   h(x,x)   \:        \nabla \cdot V(x) \: \xi(x)\:  \psi(t)\:  dtdx .
	\end{equation}   
\end{lemma}
 	\begin{proof} We see that
 		$$(\nabla_x + \nabla_y)\zeta = \delta_\lambda \left( \frac{x-y}{2}\right) \nabla  \xi\left( \frac{x+y}{2}\right)  .  $$
 		Since, for any $\eta\in \RR^N,$  
 		\begin{eqnarray}
 			\nabla \delta_\lambda  (\eta) &=&  \frac{2}{\lambda^{N}}   \frac{z}{\lambda^2}    \delta' \left( \frac{  \vert  z\vert ^2}{\lambda^2}\right) ,
 		\end{eqnarray}
 	  we have
 		\begin{eqnarray}
 			\nabla  \delta_\lambda \left( \frac{x-y}{2}\right) &=& \frac{2}{\lambda^{N}}   \frac{x-y}{2\lambda^2}    \delta' \left( \frac{ \left  \vert  x-y\right \vert ^2}{(2\lambda)^2}\right) ,\quad \hbox{ for any }x,y\in \Omega  .
 		\end{eqnarray}
 		So,  
 		\begin{eqnarray}
 			\nabla_y \zeta_\lambda (x,y) &=&   \frac{1}{2}\nabla \xi\left( \frac{x+y}{2}\right) \delta_\lambda \left( \frac{x-y}{2}\right)
 			- \frac{1}{2}  \xi\left( \frac{x+y}{2}\right) \nabla \delta_\lambda \left( \frac{x-y}{2}\right)   \\  \\  &=&
 			\nabla \xi\left( \frac{x+y}{2}\right) \delta_\lambda \left( \frac{x-y}{2}\right)  -  \frac{1}{\lambda^{N+1} } \xi \left( \frac{x+y}{2}\right)          \delta' \left( \frac{ \left  \vert  x-y\right \vert ^2}{(2\lambda)^2}\right)    \: \frac{x-y}{2\lambda }
 		\end{eqnarray}
 		and
 		\begin{eqnarray}
 			(V(x) -V(y))   \cdot \nabla_y  \zeta
 			&=& \frac{1}{2}   (V(x) -V(y))   \cdot   \nabla \xi\left( \frac{x+y}{2}\right) \delta_\lambda \left( \frac{x-y}{2}\right)  \\  \\
 			&  &-  \frac{1}{\lambda^{N} } \xi\left( \frac{x+y}{2}\right)      \delta' \left( \frac{ \left  \vert  x-y\right \vert ^2}{(2\lambda)^2}\right)    \frac{ V(x) -V(y)}{\lambda }  \cdot  \frac{x-y}{2\lambda }       .
 		\end{eqnarray}
 		This implies that
 		$$ \lim_{\lambda \to 0 } \int\!\!\int    h(x,y) \:    (V(x) -V(y))   \cdot \nabla_y  \zeta \: dxdy $$
 		$$= \lim_{\lambda \to 0 }  -  \frac{1}{\lambda^{N} } \int\!\!\int   \xi\left( \frac{x+y}{2}\right)   \delta' \left( \frac{ \left  \vert  x-y\right \vert ^2}{(2\lambda)^2}\right)      \frac{ V(x) -V(y)}{\lambda }  \cdot  \frac{x-y}{2\lambda }         \: dxdy =:  I(\lambda) .
 		$$
 		Changing the variable by setting
 		$$z= \frac{  x-y  }{2\lambda }   $$
 		we get $x= y+ 2\lambda z$ and $ dx=(2\lambda)^N\: dz ,$ so that
 		\begin{eqnarray}
 			-  I(\lambda) &=& 2^N   \int\!\!\int    h((y+2\lambda z,y)  \: \xi\left( y+\lambda z\right)    \delta'   \left( \vert z\vert ^2\right  )
 			\:      \frac{ V(y+2\lambda z) -V(y)}{\lambda }  \cdot  z\:    dzdy \\  \\
 			&=&
 			2^N     \sum_{i=1}^N      \int\!\!\int  h((y+2\lambda z,y) \: \xi\left( y+\lambda z\right)   \delta'   \left( \vert z\vert \right  )      \frac{ V_i(y+2\lambda z) -V_i(y)}{\lambda } \:     z_i\:
 			\: dzdy.
 		\end{eqnarray}
 		Letting $\lambda \to 0,$ we get
 		\begin{eqnarray}
 			\lim_{\lambda \to 0 } I(\lambda)  &=&  -2^{N+1}\:  \sum_{i=1}^N   \int\!\!\int   \xi(y)\: h((y,y) \:    \nabla V_i (y)\cdot z \: z_i    \delta'  (\vert z\vert ^2 )  \: dzdy \\  \\
 			&=& - 2^{N+1}\:  \sum_{i,j=1}^N   \int\!\!\int   \xi(y)\: h((y,y)  \:     \frac{\partial  V_i (y)}{\partial y_j} \: z_j \: z_i    \delta'  (\vert z\vert ^2 )  \: dzdy \\  \\
 			&=&  -2^{N+1}\:  \sum_{i=1}^N   \int\!\!\int   \xi(y)\: h((y,y)  \:     \frac{\partial  V_i (y)}{\partial y_j} \:   z_i^2     \delta'  (\vert z\vert ^2 )  \: dzdy
 			\\  \\
 			& = & - \frac{  2^{N+1}\:}{N}\:     \int   \xi(y)\:  h((y,y)   \:   \nabla \cdot V (y) \: d  y \:\int    \vert z\vert^2 \:   \delta'  (\vert z\vert ^2 )   \: dz  ,
 		\end{eqnarray}
 		where we use the fact that  (since $z\in B(0,1) \to \delta(\vert z\vert^2)$ is symmetric)
 		$$\int    {z_j \: z_i}  \delta' (\vert z\vert ^2)   \: dz =0,\quad \hbox{ for any }i\neq j $$
 		and
 		$$ \sum_{i=1}^N  \int    { z_i}^2 \delta' (\vert z\vert^2 )   \: dz = N \int    { z_1}^2 \delta' (\vert z\vert^2 )   \: dz =   \int    {\vert   z \vert}^2 \delta' (\vert z\vert^2 )   \: dz .  $$
 		At last, we use the fact that
 		$$\frac{  2^{N+1}\:}{N}\: \int    \vert z\vert^2 \:   \delta'  (\vert z\vert ^2 )   \: dz  =-1.$$
 		Indeed, recall that
 		$$  \int    \vert z\vert^2 \:   \delta'  (\vert z\vert ^2 )   \: dz = \int_0 ^\infty \int_{S^{N-1}}   \vert r\theta \vert^2 \:   \delta'  (\vert r\theta \vert ^2 )      r^{N-1} dr\: ds_{N-1} (\theta),$$
 		where $s_{N-1}$ is the probability measure defined on the sphere $S_{N-1}$   by
 		$$ s_{N-1}(A
)= \left\vert \left\{ rx\: :\:  r\in [0,1],\:  x\in A     \right  \} \right \vert.$$
 		This implies that
 		\begin{eqnarray}
 			\int    \vert z\vert^2 \:   \delta'  (\vert z\vert ^2 )   \: dz &=&\vert B(0,1)\vert \:  \int_0 ^1      r   ^2 \:   \delta'  (  r    ^2 )      r^{n-1} dr \\  \\
 			&=&\vert B(0,1)\vert \:  \int_0 ^1   \delta'  (  r    ^2 )      r^{n+1} dr\\  \\
 			&=&  - \frac{N}{2}\int  \delta (\vert z\vert ^2 )\: dz.
 		\end{eqnarray}
 		Thus the result.
 	\end{proof}

 	 \end{document}